\documentclass{article}

\usepackage[nonatbib,preprint]{neurips_2021}

\usepackage[utf8]{inputenc} 
\usepackage[T1]{fontenc}    
\usepackage{hyperref}       
\usepackage{url}            
\usepackage{booktabs}       
\usepackage{amsfonts}       
\usepackage{nicefrac}       
\usepackage{microtype}      
\usepackage{xcolor}         

\usepackage{amsmath}
\usepackage{amssymb}
\usepackage[ruled,lined,linesnumbered,vlined]{algorithm2e}
\usepackage{algorithmic}
\usepackage{mathrsfs}
\usepackage{graphicx}
\usepackage{subcaption}

\usepackage{ThmsLemmas}
\usepackage{bnl_style}

\newcommand{\col}[1]{#1}

\title{Stochastic Projective Splitting: Solving Saddle-Point Problems with Multiple Regularizers}

\author{%
  Patrick R. ~Johnstone\\
 Computational Science Initiative\\
 Brookhaven National Laboratory\\
patrick.r.johnstone@gmail.com
  \And
  Jonathan Eckstein \\
  Department of Management Science and Information Systems\\
  Rutgers University\\
  jeckstei@business.rutgers.edu
 \AND
  Thomas Flynn \\
 Computational Science Initiative\\
 Brookhaven National Laboratory\\
tflynn@bnl.gov
 \And
 Shinjae Yoo \\
 Computational Science Initiative\\
 Brookhaven National Laboratory\\
 sjyoo@bnl.gov
}

\begin{document}

\maketitle

\begin{abstract}
We present a new, stochastic variant of the projective splitting (PS) family
of algorithms for monotone inclusion problems.  It can solve min-max and
noncooperative game formulations arising in applications such as
robust ML without the convergence issues associated with gradient
descent-ascent, the current \emph{de facto} standard approach in such
situations.  Our proposal is the first version of PS able to use stochastic
(as opposed to deterministic) gradient oracles. It is also the first stochastic
method that can solve min-max games while easily handling multiple constraints
and nonsmooth regularizers via projection and proximal operators.  We close with numerical experiments on a
distributionally robust sparse logistic regression problem.
\end{abstract}

\section{Introduction}
Perhaps the most prominent application of optimization in ML is the empirical risk minimization problem.  However, inspired by the success of GANs~\cite{NIPS2014_5ca3e9b1},  ML practictioners have developed more complicated min-max and adversarial optimization formulations
\cite{yu2021fast,kuhn2019wasserstein,shafieezadeh2015distributionally,sinha2018certifiable,
lin2020gradient,NIPS2016_4588e674, huang2017context,wadsworth2018achieving,zhang2018mitigating,edwards2015censoring,celis2019improved}. Solving these multi-player games leads to issues not seen when minimizing a single loss function.
The competitive nature of a game leads to rotational dynamics that can cause intuitive gradient-based methods to fail to converge \cite{gidel2018a,daskalakis2018training,NEURIPS2020_ba9a56ce}.

A mathematical framework underlying both convex optimization and \col{saddle-point problems} is the \textit{monotone inclusion problem} (See \cite{ryu2016primer} for an introduction).
Methods developed for monotone inclusions will converge for \col{convex-concave} games as they are explicitly designed to handle such problems' governing dynamics.
Nevertheless,  monotone inclusion methods and theory are not well known in the ML community,
although there has been recent interest in monotone variational inequalities, which form a special case of monotone inclusions \cite{antonakopoulos2019adaptive,gidel2018a,daskalakis2018training,NEURIPS2020_ba9a56ce,
mertikopoulos2018optimistic}.

The most prevalent methods for solving min-max games in ML are
variants of \textit{gradient descent-ascent} (GDA).
This method alternates between a gradient-descent step for the minimizing player and a gradient-ascent step for the maximizing player.
Unfortunately, GDA requires additional assumptions to converge on convex-concave games, and it even fails for some simple 2D bilinear games \cite[Prop.~1]{gidel2018a}.
While there have been several approaches to modify either GDA \cite{chavdarova2021taming,grnarova2021generative,balduzzi2018mechanics} or the underlying game objective \cite{mescheder2018training,NIPS2017_7e0a0209,NIPS2017_4588e674} to ensure convergence,
this paper instead develops a method for solving monotone inclusions that can naturally handle game dynamics.

Our approach builds upon the recently proposed projective splitting (PS)
method with forward steps~\cite{johnstone2020projective}.
PS is designed specifically for solving monotone inclusions, thus does not
fall prey to the convergence issues that plague GDA, at least for \col{convex-concave} games.
PS is within the general class of projective splitting methods invented in \cite{eckstein2008family} and developed further in \cite{eckstein2009general,alotaibi2014solving,combettes2016async,eckstein2017simplified,johnstone2018convergence,
johnstone2021single,johnstone2020only}.
These methods work by
creating a separating hyperplane between the current iterate and the solution
and then moving closer to the solution by projecting the current iterate onto this hyperplane (see Section \ref{secProjSplit} for an overview).
Other than being able to natively handle game dynamics,
the primary advantage of PS is that it \textit{fully splits} problems involving an arbitrary number of regularizers and constraints.
``Full splitting''
means that the method can handle multiple regularizers and constraints through
their respective individual proximal and projection operators,
along with the smooth terms via gradients.
What makes this useful is that many of the regularizers used in ML have proximal operators that are easy to compute \cite{parikh2013proximal}.

Despite these advantages, the preexisting PS framework has a significant
drawback: it requires deterministic gradient oracles. This feature
makes it impractical for application to large datasets for which
stochastic oracles may be the only feasible option.


\paragraph{Contributions}
The primary contribution of this work is a new projective splitting algorithm that allows for a stochastic gradient oracle. We call the method \textit{stochastic projective splitting} (SPS).  It is the first stochastic method to fully split the monotone inclusion problem
\begin{align}\label{mono1}
\text{Find }z\in\rR^d
\,\,\text{ s.t. }\,\,
0 \in \sumin A_i(z) + B(z)
\end{align}
where $B$ is monotone and $L$-Lipschitz and each $A_i$ is maximal monotone and
typically set valued, usually arising from a constraint or a nonsmooth
regularizer in the underlying optimization problem or game (see for example
\cite{ryu2016primer} for definitions). It interrogates the Lipschitz operator $B$
through a stochastic oracle. Previous methods splitting this inclusion have
either required a deterministic oracle for $B$, or have made far more
restrictive assumptions on the noise or the operators \cite{briceno2011monotone+,combettes2012primal,malitsky2020forward,bot2019forward,van2021convergence}.  Our proposal is the
first stochastic method that can solve min-max problems under reasonable
assumptions, while easily handling multiple regularizers and constraints.

When moving away from a deterministic gradient oracle in projective splitting, a key difficulty is that the generated hyperplanes do not guarantee separation between the solution and the current point. We solve this issue by relaxing the projection: we only update each iterate in the \textit{direction} of the noisy projection and scale its movement by a decreasing stepsize that allows for control of the stochastic error.
Using the framework of \textit{stochastic quasi-Fej\'{e}r monotonicity} \cite{combettes2015stochastic}, we prove almost-sure convergence of the final iterate and do not require averaging of the iterates (Theorem \ref{thmMain}, Section \ref{secMainResults}). We also provide a non-asymptotic convergence rate for the approximation residual (Theorem \ref{thmConvR}, Section \ref{secMainResults}).

A special case of SPS is the recently-developed Double Stepsize Extragradient Method (DSEG) \cite{NEURIPS2020_ba9a56ce}.  When
only $B$ is present in \eqref{mono1},  DSEG and SPS coincide. Thus, our method extends DSEG to allow for regularizers and constraints.
Our analysis also provides a new interpretation for DSEG as a special case of projective splitting.
Our nonasymptotic convergence rate for SPS also applies to DSEG under no additional assumptions.  In contrast, the original convergence rate analysis for DSEG requires either strong monotonicity or an error bound.

We close with numerical experiments on a distributionally robust sparse logistic regression problem. This is a nonsmooth convex-concave min-max problem which can be converted to \eqref{mono1} with $n=2$ set-valued operators. Owing to its ability to use a stochastic oracle, SPS performs quite well  compared with deterministic splitting methods.

\paragraph{Non-monotone problems} The work \cite{NEURIPS2020_ba9a56ce} included a local convergence analysis for DSEG applied to locally monotone problems. For min-max problems, if the objective is  locally convex-concave at a solution and DSEG is initialized in close proximity, then for small enough stepsizes it converges to the solution with high probability.  It is possible to extend this result to SPS, along with our convergence rate analysis.  This result is beyond the scope of this work, but the appendix provides a proof sketch.

\section{Background on Monotone Inclusions}
\label{secBackG}
Since they are so important to SPS, this section provides some background
material regarding monotone inclusions, along with their connections to convex
optimization, games, and ML.  The appendix discusses their connections to
variational inequalities.   For a more thorough treatment, we refer to
\cite{bauschke2011convex}.

\paragraph{Fundamentals}
Let $f:\rR^d\to\rR\cup\{\infty\}$ be closed, convex, and proper (CCP).  Recall that
its \emph{subdifferential} $\partial f$ is given by
$
\partial f(x) \triangleq \{g:f(y)\geq f(x)+g^\top (y-x)\}.
$
The map $\partial f$ has the property
\begin{align*}
u\in \partial f(x),v\in \partial f(y)\implies (u - v)^\top(x - y) \geq 0,
\end{align*}
and any point-to-set map having this property is called a \emph{monotone
operator}. A minimizer of $f$ is any $x^*$ such that $0\in\partial f(x^*)$.
This is perhaps the simplest example of a \textit{monotone inclusion}, the
problem of finding $x$ such that $0 \in T(x)$, where $T$ is a monotone
operator. If $f$ is smooth, then $\partial f(x) = \{\nabla f(x)\}$ for all $x$, and
the monotone inclusion $0\in\partial f(x)$ is equivalent
to the first-order optimality condition $0 = \nabla f(x)$.


Next, suppose that we wish to minimize the sum of two CCP functions $f,g:\rR^d\to\rR\cup\{\infty\}$. Since under certain regularity conditions (\cite[Thm.~16.47]{bauschke2011convex}) it holds that
$
\partial (f+g) = \partial f + \partial g,
$
minimizing $f + g$ may be accomplished by solving the monotone inclusion
$
0\in\partial f(x) + \partial g(x).
$
The ``+" here denotes the Minkowski sum (also known as the \emph{dilation},
the set formed by collecting the sums of all pairs of points from the two sets); sums of monotone operators formed in this way are also monotone.
Constrained problems of the form $\min_{x\in\cC} f(x)$ for a
closed convex set $\cC$ are equivalent to the above formulation with $g(x) =
\iota_{\cC}(x)$, where  $\iota_{\cC}(x)$ denotes the \textit{indicator
function} returning $0$ when $x\in\cC$ and $+\infty$ otherwise. The
subdifferential of the indicator function, $\partial\iota_{\cC}$, is known as
the \emph{normal cone map} and written as $N_{\cC}$. For closed convex sets,
the normal cone map is a maximal~\cite[Def.~20.20]{bauschke2011convex}
monotone operator~\cite[Example 20.26]{bauschke2011convex}.


Under certain regularity conditions~ \cite[Cor.~16.5] {bauschke2011convex}, minimizing a sum of CCP
functions $f_1,\ldots,f_n$ is equivalent to solving the monotone inclusion
formed from the sum of their subdifferentials:
\begin{align*}
x^*\in \underset{x\in\rR^d}{\arg\min} \sumin f_i(x)
\iff
0 \in\sumin \partial f_i(x^*).
\end{align*}
Multiple constraints of the form $x \in \cap_{i=1}^c \cC_i$, where each
set $\cC_i \subseteq \rR^d$ is closed and convex, may be imposed by adding a
sum of indicator functions $\sum_{i=1}^c\iota_{\cC_i}$ to the objective.  Under standard
regularity conditions \cite[Cor.~16.5]{bauschke2011convex}), we thus have
\begin{align}\label{eqMultiReg}
x^*\in\underset{x \in \left(\bigcap_{i=1}^c \cC_i\right)}{\arg\min} \sumin f(x)
\iff
0 \in\sumin\partial f_i(x^*) + \sum_{j=1}^c N_{\cC_j}(x^*).
\end{align}

\paragraph{ML applications}
The form~\eqref{eqMultiReg} can be used to model ML problems with multiple
constraints and/or nonsmooth regularizers, including
sparse and overlapping group lasso \cite{jacob2009group},
sparse and low-rank matrix estimation problems \cite{savalle2012estimation}, and
rare feature selection~\cite{yan2020rare}. See \cite{pedregosa2018adaptive} for an overview.

\paragraph{Games} 

Consider a two-player noncooperative game
 in which each player tries to selfishly minimize its own loss, with each loss depending on the actions of both players.  Typically, the goal is to find a Nash equilibrium, in which neither player can improve its loss by changing strategy:
\begin{align}\label{defNash1}
x^* \in \underset{x\in\Theta}{\arg\min}\; F(x,y^*)
\quad\text{and}\quad
y^* \in \underset{y\in\Omega }{\arg\min}\; G(x^*,y).
\end{align}
Assuming that the admissible
strategy sets $\Theta\subseteq \rR^{d_x}$ and $\Omega\subseteq \rR^{d_y}$ are closed and convex and
that $F$ and $G$ are differentiable, the first-order necessary conditions for solving
the Nash equilibrium problem are
\begin{align}\label{gameMono}
0 \in
\left[
\begin{array}{c}
\nabla_x F(x^*,y^*)\\
\nabla_y G(x^*,y^*)
\end{array}
\right]
+
\big(
N_\Theta(x^*) \times
N_\Omega(y^*)
\big).
\end{align}
\col{If $G=-F$,  then \eqref{defNash1} is a min-max game. If in addition,
$F$ is convex in $x$ and concave in $y$} then
 $B: (x,y) \mapsto (\nabla_x
F(x,y),-\nabla_y F(x,y))^\top$ is monotone\footnote{\col{Sufficient conditions for the monotonicity of \eqref{gameMono} in the case where $G\neq-F$  are discussed in e.g.~\cite{scutari2014real,briceno2013monotone}}}
on
$\rR^{d_x+d_y}$ \cite{rockafellar1970monotone}. 
\col{In many applications, $B$ is also Lipschitz continuous.}
In this situation, \eqref{gameMono} is a monotone inclusion
involving two operators $B$ and $N_{\Theta \times \Omega}$, with $B$ being
Lipschitz. Using the simultaneous version of GDA on~\eqref{defNash1} is
equivalent to applying the forward-backward method (FB)
\cite[Thm.~26.14]{bauschke2011convex} to
\eqref{gameMono}. However, convergence of FB requires that the operator $B$ be
\textit{cocoercive} \cite[Def.~4.10]{bauschke2011convex}, and not merely
Lipschitz \cite[Thm.~26.14]{bauschke2011convex}. Thus, simultaneous GDA fails
to converge for~\eqref{defNash1} without additional assumptions (see \cite[Prop.~1]{gidel2018a} for a simple counterexample).

Regularizers and further constraints may be imposed by adding more
operators to~\eqref{gameMono}.  For example, if one wished to apply a (nonsmooth)
convex regularizer $r:\rR^{d_x} \rightarrow \rR \cup \{+\infty\}$ to the $x$
variables and a similar regularizer $d:\rR^{d_y} \rightarrow \rR \cup \{+\infty\}$
for the $y$ variables, one would add the operator $A_2 : (x,y) \mapsto
\partial r(x) \times \partial d(y)$ to the right-hand side
of~\eqref{gameMono}.


\paragraph{ML applications of games}
Distributionally robust supervised learning (DRSL) is an emerging framework
for improving the stability and reliability of ML models in the face of
distributional shifts
\cite{yu2021fast,kuhn2019wasserstein,shafieezadeh2015distributionally,sinha2018certifiable,
lin2020gradient,NIPS2016_4588e674}.  Common approaches to DRSL formulate the
problem as a min-max game between a learner selecting the model parameters and
an adversary selecting a worst-case distribution subject to some ambiguity
set around the observed empirical distribution.  This min-max problem is often
further reduced to either a finite-dimensional saddlepoint problem or a convex
optimization problem.

DRSL is a source of games with multiple constraints/regularizers. One such
formulation, based on \cite{yu2021fast}, is discussed in the experiments
below. The paper \cite{NIPS2016_4588e674} uses an amiguity set based on
$f$-divergences, while \cite{sinha2018certifiable} introduces a Lagrangian
relaxation of the Wasserstein ball. When applied to models utilizing multiple
regularizers \cite{jacob2009group,savalle2012estimation,yan2020rare}, both of
these approaches lead to min-max problems with multiple regularizers.

Other applications of games in ML, although typically nonconvex, include
generative adversarial networks
(GANs)~\cite{NIPS2014_5ca3e9b1,pmlr-v70-arjovsky17a}, fair
classification~\cite{wadsworth2018achieving,zhang2018mitigating,edwards2015censoring,celis2019improved}
, and adversarial privacy \cite{huang2017context}.

\paragraph{Resolvents,  proximal operators, and projections}
A fundamental computational primitive for
solving monotone inclusions is the \textit{resolvent}. The resolvent of a monotone operator
$A$ is defined to be
$
J_A \triangleq (I+A)^{-1}
$
where $I$ is the identity operator and the inverse of any operator $T$ is
simply $T^{-1} : x \mapsto \{y:Ty \ni x\}$.  If $A$ is maximal monotone, then
for any $\rho>0$, $J_{\rho A}$ is single valued, nonexpansive, and has domain
equal to $\rR^d$~\cite[Thm. 21.1 and Prop. 23.8]{bauschke2011convex}.
Resolvents generalize proximal operators of convex functions: the proximal
operator of a CCP function $f$ is
\begin{align*}
\text{prox}_{\rho f}(t) \triangleq \arg\min_{x\in\rR^d}\left\{\rho f(x) + (1/2)\|x - t\|^2\right\}.
\end{align*}
It is easily proved that $\prox_{\rho f} = J_{\rho\partial f}$. In turn,
proximal operators generalize projection onto convex sets: if $f = \iota_\cC$,
then $\prox_{\rho f} = \proj_\cC$ for any $\rho>0$.


In many ML applications, proximal operators, and hence resolvents, are
relatively straightforward to compute. For examples, see
\cite[Sec.~6]{parikh2013proximal}.

\paragraph{Operator splitting methods}
\emph{Operator splitting methods} attempt to solve monotone inclusions such
as~\eqref{mono1} by a sequence of operations that each involve only one of the
operators $A_1,\ldots,A_n,B$. Such methods are often presented in the context
of convex optimization problems like \eqref{eqMultiReg}, but typically apply
more generally to monotone inclusions such as \eqref{mono1}. In the specific
context of~\eqref{mono1}, each iteration of such a method ideally handles
each $A_i$ via its resolvent and the Lipschitz operator $B$ by explicit (not
stochastic) evaluation.  This is a feasible approach if the original problem
can be decomposed in such a way that the resolvents of each $A_i$ are
relatively inexpensive to compute, and full evaluations of $B$ are possible.
Although not discussed here, more general formulations in which matrices couple the
arguments of the operators can broaden the applicability of operator splitting
methods.


\section{The Projective Splitting Framework}\label{secProjSplit}
Before introducing our proposed method, we give a brief introduction to the projective splitting class of methods.
\paragraph{The extended solution set}
Projective splitting is a primal-dual framework and operates in an extended space of primal and dual variables. Rather than finding a solution to \eqref{mono1}, we find a point in the \textit{extended solution set}
\begin{align}\label{Sdef}
\cS \triangleq \left\{
(z,w_1,\ldots,w_{n+1})\in\rR^{(n+2)d}
\;\Big|\;
w_i\in A_i(z)\, \forall\, i=1,\ldots,n,
w_{n+1}=B(z),
\sum_{i=1}^{n+1} w_i=0\right\}.
\end{align}
Given $p^*=(z^*,w_1^*\ldots,w_{n+1}^*)\in\cS$, it is straightforward to see
that $z^*$ solves \eqref{mono1}. Conversely, given a solution $z^*$ to
\eqref{mono1}, there must exist $w_1^*,\ldots,w_{n+1}^*$ such that
$(z^*,w_1^*,\ldots,w_{n+1}^*)\in\cS$.

Suppose $p^*=(z^*,w_1^*\ldots,w_{n+1}^*)\in\cS$.
Since $z^*$ solves \eqref{mono1}, $z^*$ is typically referred to as a \textit{primal solution}. The vectors $w_1^*,\ldots,w_{n+1}^*$ solve a dual inclusion not described here, and are therefore called a \textit{dual solution}.
It can be shown that
$\cS$ is closed and convex; see for example \cite{johnstone2020projective}.

We will assume that a solution to \eqref{mono1} exists, therefore the set $\cS$ is nonempty.

\paragraph{Separator-projection framework}
Projective splitting methods are instances of the general
\emph{separator-projection} algorithmic framework for locating a member of a
closed convex set $\cS$ within a linear space $\cP$.  Each iteration $k$
of algorithms drawn from this framework operates by finding a set $H_k$ which
separates the current iterate $p^k \in \cP$ from $\cS$, meaning that $\cS$ is
entirely in the set and $p^k$ typically is not. One then attempts to ``move closer" to
$\cS$ by projecting the $p^k$ onto $H_k$.
%
In the particular case of projective splitting applied to the
problem~\eqref{mono1} using~\eqref{Sdef}, we select the space $\cP$ to be
\begin{align}\label{subspaceP}
  \mathcal{P} &\triangleq
  \left\{(z,w_1,\ldots,w_{n+1})\in\rR^{(n+2)d}
  \;\Big|\;
  \suminp w_i = 0\right\},
\end{align}
and each separating set $H_k$ to be the half space
$\{p\in\cP\;|\;\varphi_k(p)\leq 0\}$ generated by an affine function
$\varphi_k : \cP \to \rR$.  The general intention is to construct $\varphi_k$ such that
$\varphi_k(p^k)>0$, but $\varphi_k(p^*)\leq 0$ for all $p^*\in\cS$.  The construction employed for $\varphi_k$ in the case of~\eqref{mono1} and~\eqref{Sdef} is of the form
\begin{align}\label{sepForm}
\varphi_k(z,w_1,\ldots,w_{n+1})
&\triangleq
\sum_{i=1}^{n+1}\langle z - x_i^k,y_i^k - w_i\rangle
\end{align}
for some points $(x_i^k,y_i^k)\in\rR^{2d}$, $i=1,\ldots,n+1$, that must be
carefully chosen (see below).  Note that any function of the
form~\eqref{sepForm} must be affine when restricted to $\cP$.  As mentioned
above, the standard separator-projection algorithm obtains its next iterate
$p^{k+1}$ by projecting $p^k$ onto $H_k$.  This calculation involves the
usual projection step for a half space, namely
\begin{align}\label{projStepUpdate}
p^{k+1} = p^k - \alpha_k\nabla\varphi_k,
\quad\text{ where }\quad \alpha_k = {\varphi_k(p^k)}/{\|\nabla\varphi_k\|^2},
\end{align}
where the gradient $\nabla\varphi_k$ is computed relative to $\cP$, thus
resulting in $p^{k+1} \in \cP$ (over- or under-relaxed variants of this step
are also possible).


\section{Proposed Method}
\label{secProposed}

The proposed method is given in Algorithm \ref{algSPS} and called
\textit{Stochastic Projective Splitting} (SPS). Unlike prior versions of
projective splitting,   SPS does not employ the stepsize $\alpha_k$
of~\eqref{projStepUpdate} that places the next iterate exactly on the
hyperplane given by $\varphi_k(p)=0$.  Instead, it simply moves in the
\textit{direction} $-\nabla\varphi_k$ with a pre-defined stepsize
$\{\alpha_k\}$. This fundamental change is required to deal with the
stochastic noise on lines \ref{lineNoise1} and \ref{lineXYend}. This noise
could lead to the usual choice of $\alpha_k$ defined in
\eqref{projStepUpdate} being unstable and difficult to analyze.
In order to guarantee convergence, the parameters $\alpha_k$ and $\rho_k$ must be
chosen to satisfy certain conditions given below.
Note that the gradient is calculated with respect to the subspace $\cP$ defined in
\eqref{subspaceP}; since the algorithm is initialized within $\cP$, it remains
in $\cP$, within which $\varphi_k$ is affine.  Collectively, the updates on lines
\ref{lineProj1}-\ref{lineProj2} are equivalent to
$
p^{k+1} = p^k - \alpha_k\nabla\varphi_k,
$
where $p^k = (z^k,w_1^k,\ldots,w_{n+1}^k)$.

\begin{algorithm}[b]
{
  \DontPrintSemicolon
\SetKwInOut{Input}{Input}
\Input{$p^1 = (z^1,w_1^1,\ldots,w_{n+1}^1)$ s.t. $\suminp w_i^1 = 0$, $\{\alpha_k,\rho_k\}_{k=1}^\infty$, $\tau>0$}
\For{$k=1,2,\ldots$}
{
  \For{$i=1,\ldots,n$}
  {
    $t_i^k = z^k + \tau w_i^k$\label{lineXYone}\;
    $x_i^k = J_{\tau A_i}(t_i^k)$\label{xupdate}\;
    $y_i^k = \tau^{-1}(t_i^k - x_i^k)$\label{yupdate}\;
  }
    $r^k = B(z^k) + \epsilon^k$
        \tcp*[r]{$\epsilon^k$ is unknown noise term}\label{lineNoise1}
    $x_{n+1}^k = z^k - \rho_k(r^k - w_{n+1}^k)$ \label{xupdateLip} \;
    $y_{n+1}^k = B(x_{n+1}^k) + e^k$\tcp*[f]{$e^k$ is unknown noise term}
    \label{lineXYend}

  $z^{k+1} = z^k - \alpha_k\suminp y_i^k$ \label{lineProj1}\;
  $w_i^{k+1} = w_i^k   - \alpha_k(x_i^k - \frac{1}{n+1}\suminp x_i^k)\quad i=1,\ldots,n+1$  \label{lineProj2}
}
}
\caption{Stochastic Projective Splitting (SPS)}
\label{algSPS}
\end{algorithm}

Note that SPS does not explicitly evaluate $\varphi_k$, which is only used in
the analysis, but it does keep track of $(x_i^k,y_i^k)$ for $i=1,\ldots,n+1$.
The algorithm's memory requirements scale linearly with the number of
nonsmooth operators $n$ in the inclusion~\eqref{mono1}, with the simplest
implementation storing $(3n + 5)d$ working-vector elements.  This requirement
can be reduced to $(n + 7)d$ by using a technique discussed in the
appendix.  In most applications, $n$ will be small, for example $2$ or $3$.

\paragraph{Updating $(x_i^k,y_i^k)$}
The variables $(x_i^k,y_i^k)$ are updated on lines
\ref{lineXYone}-\ref{lineXYend} of Algorithm \ref{algSPS}, in which $e^k$ and
$\epsilon^k$ are $\rR^d$-valued random variables defined on a probability
space $(\Omega,\mbF,P)$. For $B$ we use a new, noisy version of the
two-forward-step procedure from \cite{johnstone2020projective}. For each
$A_i$, $i=1,\ldots,n$, we use the same resolvent step used in previous
projective splitting papers, originating with \cite{eckstein2008family}.  In the case $\epsilon^k = e^k = 0$,
the selection of the $(x_i^k,y_i^k)$ is identical to that
proposed in~\cite{johnstone2020projective}, resulting in the hyperplane $\{p:\varphi_k(p) = 0\}$ strictly separating
$p^k$ from $\cS$.

SPS achieves full splitting of \eqref{mono1}.  Each $A_i$ is processed
separately using a resolvent and the Lipschitz term $B$ is processed via a
stochastic gradient oracle. When the $A_i$ arise from regularizers or constraints, as
discussed in Section \ref{secBackG}, their resolvents can be readily computed so
long as their respective proximal/projection operators have a convenient
form.

\paragraph{Noise assumptions}
Let $\mbF_k\triangleq\sigma(p^1,\ldots,p^k)$ and
$\mbE_k \triangleq\sigma(\epsilon^k)$. The stochastic estimators for the gradients,
$r^k$ and $y_{n+1}^k$, are assumed to be \textit{unbiased}, that is, the noise has
mean $0$ conditioned on the past:
\begin{align}
\E[\epsilon^k|\mbF_k]=0,\quad  \E[e^k|\mbF_k]=0\quad a.s.\label{unbiasedAss}
\end{align}
We impose the following mild assumptions on the variance of the noise:
\begin{align}\label{noiseBound1}
 \E\left[ \|\epsilon^k\|^2|\mbF_k\right] &\leq N_1+N_2\|B(z^k)\|^2\quad a.s.
\\\label{noiseBound2}
\E\left[ \|e^k\|^2|\mbF_k,\mbE_k\right]&\leq N_3+N_4\|B(x_{n+1}^k)\|^2\quad a.s.,
\end{align}
where $0\leq N_1, N_2, N_3, N_4 <\infty$.
We do not require $e^k$ and $\epsilon^k$ to be independent of one another.

\paragraph{Stepsize choices}
The stepsizes $\rho_k$ and $\alpha_k$ are assumed to be deterministic. A
constant stepsize choice which obtains a non-asymptotic convergence rate will be
considered in the next section (Theorem \ref{thmConvR}).  The stepsize conditions we will impose to
guarantee almost-sure convergence (Theorem \ref{thmMain}) are
\begin{align}\label{stepRuleSumInf}
	\sumk \alpha_k\rho_k = \infty,\quad
	\sumk \alpha_k^2 <\infty,\quad
	\sumk \alpha_k\rho_k^2 <\infty,
	\,\,
\text{ and }
\,\,
	\rho_k &\leq \orho <\frac{1}{L}.
\end{align}
For example, in the case $L=1$, a particular choice which satisfies these constraints is
\begin{align*}
	\alpha_k = k^{-0.5 - p} \,\,\text{ for }\,\,0<p<0.5,\,\, \text{ and }\,\,
	\rho_k = k^{-0.5+t} \,\,\text{ for }\,\, p \leq t < 0.5p+0.25.
\end{align*}
For simplicity, the stepsizes $\tau$ used for the resolvent updates in lines
\ref{lineXYone}-\ref{yupdate} are fixed, but they could be allowed to vary
with both $i$ and $k$ so long as they have finite positive lower and
upper bounds.

\section{Main Theoretical Results}\label{secMainResults}

\begin{theorem}\label{thmMain}
For Algorithm \ref{algSPS}, suppose \eqref{unbiasedAss}-\eqref{stepRuleSumInf} hold.
Then with probability one it holds that $z^k\to z^*$, where $z^*$ solves \eqref{mono1}.
\end{theorem}

\paragraph{Proof sketch}
Theorem \ref{thmMain} is proved in the appendix, but we provide a brief sketch here.
The proof begins by deriving a simple recursion inspired by the analysis of SGD \cite{robbins1951stochastic}.
Since
$
	p^{k+1} = p^k - \alpha_k\nabla\varphi_k,
$
a step of projective splitting can be viewed as GD applied to the affine hyperplane
generator function $\varphi_k$. Thus, for any  $p^*\in\cP$,
\begin{align}
	\|p^{k+1} - p^*\|^2
	&=
	\|p^k - p^*\|^2 - 2\alpha_k\langle \nabla\varphi_k,p^k - p^*\rangle + \alpha_k^2\|\nabla\varphi_k\|^2
	\nonumber\\\label{eqStart}
	&=
		\|p^k - p^*\|^2 - 2\alpha_k(\varphi_k(p^k) - \varphi_k(p^*)) + \alpha_k^2\|\nabla\varphi_k\|^2,
\end{align}
where in the second equation we have used that $\varphi_k(p)$ is affine on $\cP$.
The basic strategy is to show that, for any $p^*\in\cS$,
\begin{align*}
\E[\|\nabla\varphi_k\|^2|\mbF_k] \leq C_1\|p^k - p^*\|^2 + C_2  \quad a.s.
\end{align*}
for some $C_1, C_2 > 0$. This condition allows one to establish stochastic
quasi-Fej\'{e}r monotonicity (SQFM) \cite[Proposition
2.3]{combettes2015stochastic} of the iterates to $\cS$. One consequence of
SQFM is that with probability one there exists a subsequence $v_k$ such that
$\varphi_{v_k}(p^{v_k}) -
\varphi_{v_k}(p^*)$ converges to $0$.
Furthermore, roughly speaking, we will show that
$
\varphi_{k}(p^{k}) -
\varphi_{k}(p^*)$
provides an upper bound on the
following ``approximation residual" for SPS:
\begin{align}
O_k
\triangleq
\sumin \|y_i^k - w_i^k\|^2
+\sumin \|z^k - x_i^k\|^2
+ \| B (z^k) - w_{n+1}^k\|^2
.\label{Okdef}
\end{align}
$O_k$ provides an approximation error for SPS, as formalized in the following lemma:
\begin{lemma}
For SPS, $p^k=(z^k,w_1^k,\ldots,w_{n+1}^k)\in\cS$ if and only if $O_k=0$.\label{lemOk}
\end{lemma}
\vspace{-1.5ex}
Since $y_i^k\in A_i(x_i^k)$ for $i=1,\ldots,n$, having $O_k=0$ implies that $z^k =
x_i^k$, $w_i^k = y_i^k$, and thus $w_i^k\in A_i(z^k)$ for $i=1,\ldots,n$.
Since $w_{n+1}^k = B(z^k)$ and $\sum_{i=1}^{n+1}w_i^k =
0$, it follows that $z^k$ solves \eqref{mono1}.  The reverse direction is proved
in the appendix.

The quantity $O_k$ generalizes the role played by the norm of the gradient in
algorithms for smooth optimization.  In particular, in the special case where
$n=0$ and $B(z)=\nabla f(z)$ for some smooth convex function $f$, one has $O_k
= \|\nabla f(z^k)\|^2$.

Combining the properties of $O_k$ with other results following from SQFM (such as
boundedness) will allow us to derive almost-sure convergence of the iterates
to a solution of \eqref{mono1}.

\paragraph{Convergence rate}
\label{secConvRate}

We can also establish non-asymptotic convergence rates for the approximation
residual $O_k$:
\begin{theorem}\label{thmConvR}
Fix the total iterations $K\geq 1$ of Algorithm \ref{algSPS} and
set
\begin{align}\label{step1}
\forall k=1,\dots, K: \rho_k=\rho\triangleq
\min
\left\{
K^{-1/4},\frac{1}{2L}
\right\}
\quad\text{ and }\quad \alpha_k =
C_f \rho^2
\end{align}
for some $C_f>0$.
Suppose
\eqref{unbiasedAss}-\eqref{noiseBound2} hold.
Then
$$
\frac{1}{K}\sum_{j=1}^K
\E[O_j]
=
\bigO(K^{-1/4})
$$
where the constants are given (along with the proof) in the appendix.
\end{theorem}


Theorem \ref{thmConvR} implies that if we pick an iterate $J$ uniformly at random
from $1,\ldots,K$, then the expected value of $O_J$ is $\bigO(K^{-1/4})$.
As far as we know, this is the first convergence rate for a stochastic full-splitting method solving \eqref{mono1}, and it is not clear whether it can be reduced, either by a better analysis or a better method.  Faster rates are certainly possible for deterministic methods; Tseng's method obtains $\bigO(K^{-1})$ rate \cite{monteiro2010complexity}.
Faster rates are also possible for stochastic methods under \textit{strong} monotonicity and when $n=0$ \cite{kannan2019optimal,NEURIPS2020_ba9a56ce}.  Faster \textit{ergodic} rates for stochastic methods have been proved for special cases with $n=1$ with a compact constraint \cite{juditsky2011solving}.
What is needed is a tight lower bound on the convergence rate of any first-order splitting method applied to \eqref{mono1}.
Since nonsmooth convex optimization is a special case of \eqref{mono1}, lower bounds for that problem apply \cite{nemirovskij1983problem}, but they may not be tight for the more general monotone inclusion problem.

\section{Related Work}
Arguably the three most popular classes of operator splitting algorithms are
forward-backward splitting (FB) \cite{combettes2011proximal}, Douglas-Rachford
splitting (DR) \cite{lions1979splitting}, and Tseng's method
\cite{tseng2000modified}. The extragradient method (EG) is similar to Tseng's
method, but has more projection steps per iteration and only applies to
variational inequalities
\cite{korpelevich1977extragradient,nemirovski2004prox}. The popular
Alternating Direction Method of Multipliers (ADMM), in its standard form, is a
dual application of DR \cite{gabay1983chapter}.  FB, DR, and Tseng's method
apply to monotone inclusions involving two operators, with varying assumptions
on one of the operators. It is possible to derive splitting methods for the
more complicated inclusion \eqref{mono1}, involving more than two operators,
by applying Tseng's method to a product-space reformulation
\cite{briceno2011monotone+,combettes2012primal} (for more on the product-space
setting, see the appendix). The recently developed forward-reflected-backward
method \cite{malitsky2020forward} can be used in the same way.  The
three-operator splitting method \cite{davis2015three} can only be applied to
\eqref{mono1} if $B$ is cocoercive rather than merely Lipchitz, and thus its
usefulness is mostly limited to optimization applications and not games.

The above-mentioned methods are all deterministic, but stochastic operator
splitting methods have also been developed.  The preprint
\cite{bot2019forward} develops a stochastic version of Tseng's method under
the requirement that the noise variance goes to $0$. In ML, this could be achieved with
the use of perpetually increasing batch sizes,
a strategy that is impractical in many scenarios.  The
stochastic version of FRB proposed in \cite{van2021convergence} has more
practical noise requirements, but has stronger assumptions on the problem
which are rarely satisfied in ML applications: either uniform/strong
monotonicity or a bounded domain.  The papers  \cite{NIPS2016_5d6646aa} and
\cite{pedregosa2019proximal} consider stochastic variants of three-operator
splitting, but they can only be applied to optimization problems.  The methods of
\cite{zhao2018stochastic} and \cite{bohm2020two} can be applied to simple
saddle-point problems involving a single regularizer.

There are several alternatives to the (stochastic) extragradient method that
reduce the number of gradient evaluations per iteration from two to one
\cite{NEURIPS2019_4625d8e3,malitsky2020forward,gidel2018a}. However,  these
methods have more stringent stepsize limits, making it unclear \emph{a priori}
whether they will outperform two-step methods.

DSEG is a stochastic version of EG \cite{NEURIPS2020_ba9a56ce}.  The primary
innovation of DSEG is that it uses different stepsizes for the
extrapolation and update steps, thereby resolving some of the convergence
issues affecting stochastic EG. As noted earlier, DSEG is the special case of
our SPS method in which $n=0$, that is, no regularizers/constraints are present in
the underlying game.  The analysis in \cite{NEURIPS2020_ba9a56ce} also did not
consider the fixed stepsize choice given in Theorem \ref{thmConvR}.


\section{Experiments}\label{secExps}
We now provide some numerical results regarding the performance of SPS as
applied to distributionally robust supervised learning (DRSL). We follow the
approach of \cite{yu2021fast}, which introduced a min-max formulation of
Wasserstein DRSL.  While other approaches reduce the problem to convex
optimization, \cite{yu2021fast} reduces it to a finite-dimensional min-max
problem amenable to the use of stochastic methods on large datasets.  However,
unlike our proposed SPS method, the variance-reduced extragradient method that
\cite{yu2021fast} proposes cannot handle multiple nonsmooth regularizers or
constraints on the model parameters.

Consequently, we consider distributionally robust sparse logistic regression
(DRSLR),  a problem class equivalent to that considered in \cite{yu2021fast},
but with an added $\ell_1$ regularizer, a standard tool to induce
sparsity.  We solve the following convex-concave min-max problem:
\renewcommand{\arraystretch}{1.4}
\begin{align}
\begin{array}{rl}
\displaystyle{\min_{\substack{\beta\in\rR^d \\ \lambda\in\rR\,\,\,}}} \;\;
\displaystyle{\max_{\gamma\in\rR^m}}
&
\displaystyle{
\left\{
\lambda(\delta - \kappa) +
\frac{1}{m}\sum_{i=1}^m\Psi(\langle \hat{x}_i,\beta\rangle)
+
\frac{1}{m}
\sum_{i=1}^m
\gamma_i(
\hat{y}_i\langle\hat{x}_i,\beta\rangle - \lambda\kappa
)
+
c\|\beta\|_1
\right\}
}
\\
\,\text{s.t.} &
\|\beta\|_2\leq \lambda/(L_\Psi+1) \qquad \|\gamma\|_\infty\leq 1.
\end{array}
\label{drslr}
\end{align}
This model is identical to that of~\cite[Thm. 4.3]{yu2021fast} except for the
addition of the $\ell_1$ regularization term $c\|\beta\|_1$, where $c\geq 0$
is a given constant.  The goal is to learn the model weights $\beta$ from a
training dataset of $m$ feature vectors $\hat{x}_i$ and corresponding labels
$\hat{y}_i$.  Rather than computing the expected loss over the training set,
the formulation uses, for each $\beta$, the worst possible distribution within
a Wasserstein-metric ball around the empirical distribution of the
$\{(\hat{x}_i,\hat{y}_i)\}$, with the parameter $\delta\geq 0$ giving the
diameter of the ball and the parameter $\kappa\geq 0$ specifying the relative
weighting of features and labels.  The variables $\gamma$ and $\lambda$
parameterize the selection of this worst-case distribution in response to the
model weights $\beta$.  Finally, $\Psi$ is the logistic loss kernel $t \mapsto
\log(e^t+e^{-t})$ and $L_\Psi=1$ is the corresponding Lipschitz constant.

We converted~\eqref{drslr} to the form~\eqref{mono1} with $n=2$, with the
operator $A_1$ enforcing the constraints, $A_2$ corresponding to the objective
term $c\|\beta\|_1$, and $B$ being the vector field corresponding to the gradients of the remaining elements of the
objective. More details of the formulation are provided in the appendix.

We compared our SPS method to some deterministic methods for solving
\eqref{drslr} for a collection of real datasets from the LIBSVM repository (released under the 3-clause BSD license)
 \cite{CC01a}.
In all the experiments, we set $\delta=\kappa=1$ and $c=10^{-3}$. We
implemented SPS with
$
\alpha_k = C_d k^{-0.51}
$
and
$
\rho_k = C_d k^{-0.25}
$
and called it \textit{SPS-decay}. We also implement SPS with the fixed stepsize given in
\eqref{step1} and called it \textit{SPS-fixed}.
We compared the method to deterministic projective splitting
\cite{johnstone2020projective}, Tseng's method
\cite{tseng2000modified,combettes2012primal}, and the
forward-reflected-backward method \cite{malitsky2020forward} (FRB). To the
best of our knowledge,  there is no stochastic method besides SPS capable of
solving \eqref{drslr} under standard assumptions. We show results for three
LIBSVM standard datasets: \textit{epsilon}\footnote{original data source
\url{http://largescale.ml.tu-berlin.de/instructions/}} ($m=4\cdot 10^5$,
$d=2000$), \textit{SUSY} \cite{baldi2014searching,Dua:2019} ($m=2\cdot
10^6$, $d=18$), and \textit{real-sim}\footnote{Original data source
\url{https://people.cs.umass.edu/~mccallum/data.html}} ($m=72,\!309$,
$d=20,\!958$). For SPS-fixed, we tuned $C_f$, arriving at $C_f=1$ for
epsilon and real-sim, and  $C_f=5$ for SUSY.
For SPS-decay, we tune $C_d$ arriving at $C_d=1$ for epsilon and SUSY, and $C_d=0.5$ for real-sim.
For SPS, we use
a batchsize of $100$. All methods are initialized at the same random point.

To measure the progress of the algorithms, we used the ``approximation residual''
\begin{align}\label{defRk}
R_k &\triangleq
\textstyle{
\sumin \|z^k - x_i^k\|^2 + \big\| B(z^k) + \sumin y_i^k \big\|^2.
}
\end{align}
This measure is related to $O_k$ but does not involve the dual iterates
$w_i^k$.
As with $O_k$, having $R_k=0$ implies that $z^k$ solves \eqref{mono1}.
We use $R_k$ instead of $O_k$ because it is also possible to compute essentially the
same measure of convergence from the iterates of the other tested algorithms,
providing a fair comparison. The appendix provides the details of the derivation
of the residual measure from each algorithm and explores the relationship
between $R_k$ and $O_k$.


Figure \ref{fig} plots the approximation residual versus running time for all
five algorithms under consideration.  The computations were performed using
Python 3.8.3 and \texttt{numpy} on a 2019 MacBook Pro with a 2.4GHz 8-core Intel I9 processor and 32GB of RAM . Being a stochastic method, SPS-decay seems to outperform the
deterministic methods at obtaining a medium-accuracy solution quickly.
Overall,  SPS-decay outperforms SPS-fixed.

\begin{figure}
\centering
\begin{subfigure}{.33\textwidth}
  \centering
  \includegraphics[width=\linewidth]{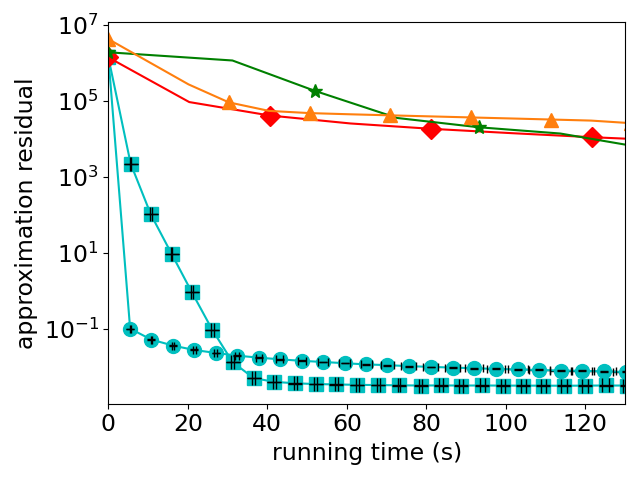}
  \label{fig:sub1}
\end{subfigure}%
\begin{subfigure}{.33\textwidth}
  \centering
  \includegraphics[width=\linewidth]{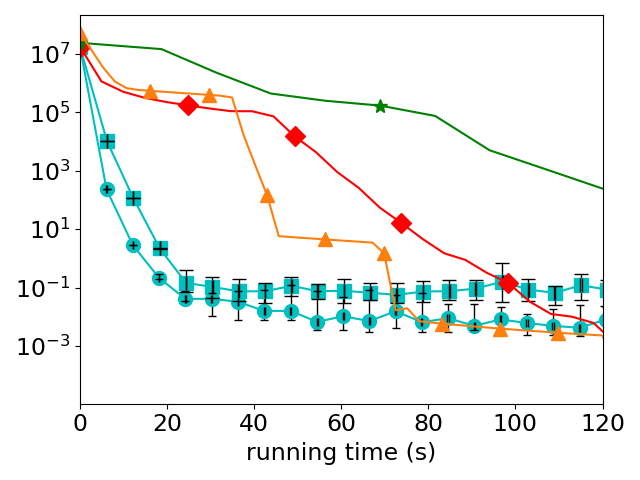}
  \label{fig:sub2}
\end{subfigure}
\begin{subfigure}{.33\textwidth}
  \centering
  \includegraphics[width=\linewidth]{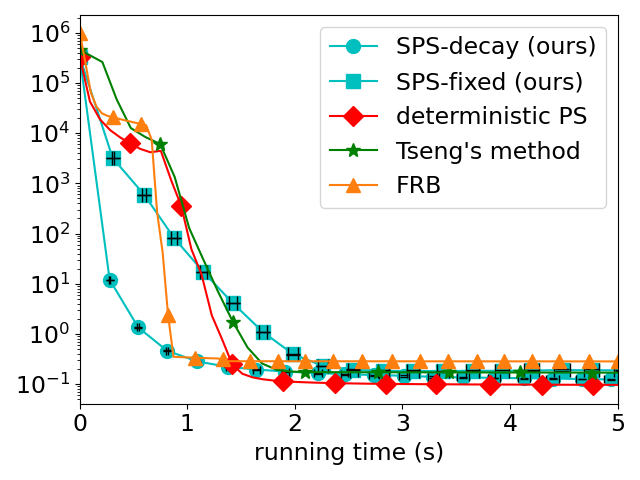}
  \label{fig:sub3}
\end{subfigure}
\vspace{-3ex}
\caption{Approximation residual versus running time for three LIBSVM benchmark
datasets, with the markers at 10-iteration intervals. Left: epsilon, middle:
SUSY, right: real-sim. Since SPS is stochastic, we plot the median results
over $10$ trials, with unit standard deviation horizontal error bars for the
running time and the vertical error bars displaying the min-to-max range of
the approximation residual. }
\label{fig}
\end{figure}

\section{Conclusions and Future Work}

We have developed the first stochastic splitting method that can handle
min-max problems with multiple regularizers and constraints.  Going
forward, this development should make it possible to incorporate regularizers
and constraints into adversarial formulations trained from large datasets. We
have established almost-sure convergence of the iterates to a solution, proved
a convergence rate result, and demonstrated promising empirical performance on a
distributionally robust learning problem.

Recent versions of deterministic projective
splitting~\cite{combettes2016async,johnstone2020projective} allow for
asynchronous and incremental operation, meaning that not all operators need to
be activated at every iteration, with some calculations proceeding with stale
inputs. Such characteristics make projective splitting well-suited to
distributed implementations. Many of our SPS results may be extended to allow
for these variations, but we leave those extensions to future work.

\section{Broader Impact} \label{secBI}
This work does not present any foreseeable societal consequence.

\bibliographystyle{spmpsci}
\bibliography{refs}

\newcounter{includeChecklist}
\setcounter{includeChecklist}{0}

\ifnum\value{includeChecklist}=1
{
\section*{Checklist}

\begin{enumerate}

\item For all authors...
\begin{enumerate}
  \item Do the main claims made in the abstract and introduction accurately reflect the paper's contributions and scope?
	\answerYes{}.
  \item Did you describe the limitations of your work?
	\answerYes{}
	See the second paragraph of Section \ref{secProposed} on memory complexity. See the discussion after Theorem \ref{thmConvR} on convergence rates. See Section \ref{secBackG} on ML applications of games for a discussion on nonconvexity.

  \item Did you discuss any potential negative societal impacts of your work?
\answerYes{} See Section \ref{secBI} (The work does not present any forseeable societal impact).
  \item Have you read the ethics review guidelines and ensured that your paper conforms to them?
    \answerYes{}
\end{enumerate}

\item If you are including theoretical results...
\begin{enumerate}
  \item Did you state the full set of assumptions of all theoretical results?
     \answerYes{}
	\item Did you include complete proofs of all theoretical results?
    \answerYes{} Proofs are in the appendix (supplementary material)
\end{enumerate}

\item If you ran experiments...
\begin{enumerate}
  \item Did you include the code, data, and instructions needed to reproduce the main experimental results (either in the supplemental material or as a URL)?
    \answerYes{} Code is in the supplementary material ZIP file.
  \item Did you specify all the training details (e.g., data splits, hyperparameters, how they were chosen)?
    \answerYes{} In both Section \ref{secExps}  and the appendix.
	\item Did you report error bars (e.g., with respect to the random seed after running experiments multiple times)?
    \answerYes{} See Figure \ref{fig}.
	\item Did you include the total amount of compute and the type of resources used (e.g., type of GPUs, internal cluster, or cloud provider)?
    \answerYes{} See Section \ref{secExps}.
\end{enumerate}

\item If you are using existing assets (e.g., code, data, models) or curating/releasing new assets...
\begin{enumerate}
  \item If your work uses existing assets, did you cite the creators?
    \answerYes{} we cite LIBSVM along with the specific sources of the data.
  \item Did you mention the license of the assets?
    \answerYes{} See Section \ref{secExps} for LIBSVM license
  \item Did you include any new assets either in the supplemental material or as a URL?
    \answerNA{} We did not include any new assets
  \item Did you discuss whether and how consent was obtained from people whose data you're using/curating?
    \answerNA{} This is not relevant to the LIBSVM datasets we used
  \item Did you discuss whether the data you are using/curating contains personally identifiable information or offensive content?
    \answerNA{} This is not relevant to the LIBSVM datasets we used
\end{enumerate}

\item If you used crowdsourcing or conducted research with human subjects...
\begin{enumerate}
  \item Did you include the full text of instructions given to participants and screenshots, if applicable?
    \answerNA{}
  \item Did you describe any potential participant risks, with links to Institutional Review Board (IRB) approvals, if applicable?
    \answerNA{}
  \item Did you include the estimated hourly wage paid to participants and the total amount spent on participant compensation?
    \answerNA{}
\end{enumerate}

\end{enumerate}
}
\fi


\appendix

\section{Proof of Theorem \ref{thmMain}}

\subsection{Stochastic Quasi-Fejer Monotonicity}
The key to the analysis is showing that the algorithm satisfies \textit{Stochastic Quasi-Fejer Monotonicity} \cite{combettes2015stochastic}.

\begin{lemma}[\cite{combettes2015stochastic}, Proposition 2.3]\label{lemQuasi}
Suppose
$p^k$ is a sequence of $\rR^d$-valued random variables defined on a probability space $(\Omega,\mbF,P)$.  Let $\mbF_k = \sigma(p^1,\ldots,p^k)$. Let $F$ be a nonempty, closed subset of $\rR^d$. Suppose that, for every $p\in F$, there exists $\chi^k(p)\geq 0,\eta^k(p)\geq 0,\nu^k(p) \geq 0$ such that $\sumk\chi^k(p)<\infty$, $\sumk\eta^k(p)<\infty$
and
\begin{align*}
(\forall k\in\mathbb{N})\quad \E[\|p^{k+1}-p\|^2|\mbF_k]
\leq
(1+\chi^k(p))\|p^k - p\|^2
-\nu^k(p)
+\eta^k(p).
\end{align*}
Then the following hold:
\begin{enumerate}
\item $(\forall p\in F):\quad \sumk\nu^k(p)<\infty$ a.s.
\item $p^k$ is  bounded a.s.
\item There exists $\tilde{\Omega}$ such that $P[\tilde{\Omega}]=1$ and
$\big\{\|p^k(\omega) - p\|\big\}$ converges for every $\omega\in\tilde{\Omega}$ and $p\in F$.
\end{enumerate}
\end{lemma}

\subsection{Important Recursion for SPS}

The following lemma summarizes the key recursion satisfied by Algorithm \ref{algSPS}, to which we will apply Lemma \ref{lemQuasi}.
Recall that $L$ is the Lipschitz constant of $B$. 
\begin{lemma}
For Algorithm \ref{algSPS}, suppose \eqref{unbiasedAss}--\eqref{noiseBound2} hold
and
\begin{align}\label{upperStepByL}
\rho_k &\leq \orho < {1}/{L}.
\end{align}
Let
\begin{align*}
T_k\triangleq
\frac{\tau}{\orho}\sumin\|y_i^k - w_i^k\|^2 + \frac{1}{\orho\tau}\sumin\|z^k - x_i^k\|^2
+
2(1-\orho L)\| B (z^k) - w_{n+1}^k\|^2
\end{align*}
then for all $p^*\in\cS$, with probability one
\begin{align}
	\E[\|p^{k+1} - p^*\|^2 |\mbF_k]
	&\leq
	(1+ C_1\alpha_k^2+ C_3\alpha_k\rho_k^2)\|p^k - p^*\|^2 - \alpha_k \rho_k T_k
	+ C_2 \alpha_k^2 + C_4\alpha_k\rho_k^2\label{finalStochquasLem}
\end{align}
where $C_1,\ldots,C_4$ are nonegative constants defined in \eqref{defC1}, \eqref{defC2}, \eqref{defC3},  and \eqref{defC4} below, respectively.
\label{lemKeyRecurse}
\end{lemma}

Note that $T_k$ is a scaled version of the approximation residual $O_k$ defined
in~\eqref{Okdef}.

We proceed to first prove Lemma \ref{lemKeyRecurse} and then exploit the
implications of Lemma \ref{lemQuasi}.  Referring to~\eqref{noiseBound1}
and~\eqref{noiseBound2}, let $N\triangleq \max_{j=1,\ldots,4} N_j$. To
simplify the constants, we will use $N$ in place of $N_j$ for the noise
variance bounds given in \eqref{noiseBound1}-\eqref{noiseBound2}.

\subsection{Upper Bounding the Gradient}
\label{secUpperGrad}
Throughout the analysis, we fix some $p^*=(z^*,w_1^*\ldots,w_{n+1}^*)\in\cS$.
All statements are with probability one (almost surely), but for brevity we
will omit this unless it needs to be emphasized.

In this section, we derive appropriate upper bounds for $\|\nabla\varphi_k\|^2$ to use in \eqref{eqStart}.
We begin with $\nabla_z\varphi_k$:
\begin{align*}
	\|\nabla_z\varphi_k\|^2
	=
	\Big\|\suminp y_i^k\Big\|^2
	\leq
	2\|y_{n+1}^k\|^2 +2\Big\|\sumin y_i^k\Big\|^2
	&=
	2\big\|B(x_{n+1}^k)+e^k\big\|^2 +2\Big\|\sumin y_i^k\Big\|^2
	\\
	&\leq
   4\|B(x_{n+1}^k)\|^2 +2\Big\|\sumin y_i^k\Big\|^2
	+4\|e^k\|^2.
\end{align*}
Now next take expectations with respect to $\mbF_k$ and $\mbE_k$, and use the
bound on the variance of the noise in \eqref{noiseBound2}, obtaining
\begin{align*}
\E\left[\|\nabla_z\varphi_k\|^2|\mbF_k,\mbE_k\right]
&\leq
\E\left[
   4\|B(x_{n+1}^k)\|^2 +2\Big\|\sumin y_i^k\Big\|^2
	+4\|e^k\|^2
\; \Big| \; \mbF_k,\mbE_k\right]
\nonumber\\
&\leq
   4(N+1)\|B(x_{n+1}^k)\|^2 +2\Big\|\sumin y_i^k\Big\|^2
	+4N,
\end{align*}
where we have used that $y_i^k$ is $\mbF_k$-measurable for $i=1,\ldots,n$.
Thus, taking  expectations over $\mbE_k$ conditioned on $\mbF_k$ yields
\begin{align}\label{eq0o}
\E\left[\|\nabla_z\varphi_k\|^2|\mbF_k\right]
\leq
   4(N+1)\E[\|B(x_{n+1}^k)\|^2|\mbF_k] +2\Big\|\sumin y_i^k\Big\|^2
	+4N.
\end{align}

We will now bound the two terms on the right side of \eqref{eq0o}.
\subsubsection{First Term in \eqref{eq0o}}
First, note that
\begin{align}
\|B(z^k)\|^2
&=
\|B(z^k)- B(z^*) +  B(z^*))\|^2
\nonumber\\
&\leq
2\|B(z^k)- B(z^*)\|^2 +2\|B(z^*)\|^2
\nonumber\\
&\leq
2 L^2\|z^k-z^*\|^2 +2\| B(z^*)\|^2
\nonumber\\\label{normGradz}
&\leq
2 L^2\|p^k-p^*\|^2 +2\|B(z^*)\|^2.
\end{align}

Now, returning to the first term on the right of \eqref{eq0o}, we have
\begin{align}
	\|B(x_{n+1}^k)\|^2
	&=
	\| B(z^k) + B(x_{n+1}^k) -B(z^k) \|^2
	\nonumber\\
	&\leq
	2\| B(z^k)\|^2 + 2\|B(x_{n+1}^k) - B(z^k) \|^2
	\nonumber\\
	&\leq
	2\| B(z^k)\|^2 + 2L^2\|x_{n+1}^k - z^k \|^2
	\nonumber\\\label{eqAbove}
	&\leq
	4 L^2\|p^k-p^*\|^2 +4\|B(z^*)\|^2+ 2L^2\|x_{n+1}^k - z^k \|^2
\end{align}
where we have used \eqref{normGradz} to obtain \eqref{eqAbove}.

For the third term in \eqref{eqAbove}, we have from the calculation on
line~\ref{xupdateLip} of the algorithm that
\begin{align*}
	x_{n+1}^k - z^k &= -\rho_k(r^k - w_{n+1}^k)
	=
	-\rho_k( B(z^k)+\epsilon^k - w_{n+1}^k),
\end{align*}
and therefore
\begin{align*}
	\|x_{n+1}^k - z^k \|^2
	&=
	\rho_k^2\| B(z^k)+\epsilon^k - w_{n+1}^k\|^2
	\nonumber\\
	&\leq
		\orho^2\| B(z^k)+\epsilon^k - w_{n+1}^k\|^2
	\nonumber\\
	&\leq
	3\orho^2(\| B(z^k)\|^2+\|\epsilon^k\|^2 + \|w_{n+1}^k\|^2).
\end{align*}
We next take expectations conditioned on $\mbF_k$ and use the noise variance
bound \eqref{noiseBound1} to obtain
\begin{align*}
\E\big[\|x_{n+1}^k - z^k \|^2 \,|\, \mbF_k\big]
&\leq
\E\!\left[
	3\orho^2\big(\| B(z^k)\|^2+\|\epsilon^k\|^2 + \|w_{n+1}^k\|^2\big)
\,|\,\mbF_k\right]
\\
&\leq
	3\orho^2\big((N+1)\| B(z^k)\|^2 + \|w_{n+1}^k\|^2 + N\big).
\end{align*}
Therefore
\begin{align}
\E\big[\|x_{n+1}^k - z^k \|^2\,|\,\mbF_k\big]
	&\leq
	6\orho^2\big((N+1)\| B(z^k)\|^2 + \|w_{n+1}^k-w_{n+1}^*\|^2+\|w_{n+1}^*\|^2\big)
	+3\orho^2 N
	\nonumber\\
  &=
  6\orho^2\Big(
  2 (N+1)L^2\|p^k-p^*\|^2 +2(N+1)\|B(z^*)\|^2
  \nonumber\\
  &\qquad\qquad\qquad\qquad\qquad
  + \|w_{n+1}^k-w_{n+1}^*\|^2+\|B(z^*)\|^2
  \Big)
  +3\orho^2 N
  \nonumber\\
	&\leq
	6\orho^2\big(
	2(N+1)L^2\|p^k - p^*\|^2
	+ \|w_{n+1}^k-w_{n+1}^*\|^2\big)
		\nonumber\\
	&\quad\quad
	+18\orho^2(N+1)\|B(z^*)\|^2
		+3\orho^2 N
	\nonumber\\
	&\leq
		18\orho^2(N+1)\big(
	(L^2 +1)\|p^k - p^*\|^2
+\| B(z^*)\|^2\big)
	+3\orho^2 N
			\label{eqLast}
\end{align}
where in the equality uses \eqref{normGradz} and $w_{n+1}^*=B(z^*)$.
Combining \eqref{eqAbove} and \eqref{eqLast}, we arrive at
\begin{align}
		\E\left[\left.\left\|B(x_{n+1}^k)\right\|^2\,\right|\mbF_k\right]
		&\leq
		4 L^2
		\big[1+9\orho^2 (L^2 +1)(N+1)\big]\|p^k-p^*\|^2
			\nonumber\\
		&\quad
		+
		4\big(1+9\orho^2L^2(N+1)\big)\| B(z^*)\|^2
		+6\orho^2L^2 N.
		\label{eqAlso}
\end{align}

\subsubsection{Second term in \eqref{eq0o}}
For $i=1\ldots,n$, line~\ref{yupdate} of the algorithm may be rearranged into
$
y_i^k = \tau^{-1}(z^k- x_i^k) + w_i^k,
$
so
\begin{align}
\Big\|\sumin y_i^k\Big\|^2
&=
\Big\| \sumin(\tau^{-1}(z^k- x_i^k) + w_i^k)\Big\|^2
\nonumber\\
&\leq
2\Big\|\tau^{-1} \sumin(z^k- x_i^k)\Big\|^2 +2\Big\| \sumin w_i^k\Big\|^2
\nonumber\\
&\leq
2n\tau^{-2}\sumin\|z^k - x_i^k\|^2
+
2\Big\|\sumin w_i^k\Big\|^2
\nonumber\\
&\leq
4n^2\tau^{-2}\|z^k - z^*\|^2 + 4n\tau^{-2}\sumin\|z^* -  x_i^k\|^2
+
4n\sumin\|w_i^k - w_i^*\|^2 + 4\Big\|\sumin w_i^*\Big\|^2
\nonumber\\\label{eqBzero}
&\leq
4n^2(\tau^{-2}+1)\|p^k - p^*\|^2 + 4n\tau^{-2}\sumin\|z^* -  x_i^k\|^2
+ 4\Big\|\sumin w_i^*\Big\|^2.
\end{align}
By the definition of the solution set $\cS$ in \eqref{Sdef}, $w_i^*\in
A_i(z^*)$, so
$
z^* + \tau w_i^* \in (I+\tau A_i)(z^*),
$
and  since the resolvent is single-valued~\cite[Cor. 23.9]{bauschke2011convex} we therefore obtain
\begin{align*}
z^*  =
(I+\tau A_i)^{-1}(I+\tau A_i)(z^*)
=
J_{\tau A_i}(z^* + \tau w_i^*).
\end{align*}
From lines~\ref{lineXYone} and~\ref{xupdate} of the algorithm, we also have
$
x_i^k = J_{\tau A_i}(z^k +\tau w_i^k)
$
for $i=1\ldots,n$. Thus, using the nonexpansiveness of the resolvent~\cite[Def.
4.1 and Cor. 23.9]{bauschke2011convex}, we have
\begin{align}
\sumin
\|z^*-x_i^k \|^2
&=
\sumin
\big\|J_{\tau A_i}(z^k +\tau w_i^k) - J_{\tau A_i}(z^* + \tau w_i^*)\big\|^2
\nonumber\\
&\leq
\sumin
\|z^k +\tau w_i^k - z^* - \tau w_i^*\|^2
\nonumber\\
&=
\sumin
\|z^k - z^*+ \tau (w_i^k - w_i^*)\|^2
\nonumber\\
&\leq
2n
\|z^k - z^*\|^2+ 2\tau^2\sumin\| w_i^k - w_i^*\|^2
\nonumber\\\label{eqinBone}
&\leq
2(n+\tau^2)\|p^k - p^*\|^2.
\end{align}
Combining \eqref{eqBzero} and \eqref{eqinBone} yields
\begin{align}
\Big\|\sumin y_i^k\Big\|^2
\leq
12 n^2 \tau^{-2}(n+\tau^2)
\|p^k - p^*\|^2
+
4\Big\|\sumin w_i^*\Big\|^2.
\label{eqBtwo}
\end{align}

Combining  \eqref{eqAlso} and \eqref{eqBtwo} with \eqref{eq0o} yields
\begin{align}
\E\big[\|\nabla_z\varphi_k\|^2\,|\,\mbF_k\big]&\leq
24\left[
		(1+9\orho^2) (L^2 +1)^2(N+1)^2
+
n^2 \tau^{-2}(n+\tau^2)
\right]
		\|p^k-p^*\|^2
			\nonumber\\
		&\qquad
		+
		16(N+1)\big(1+9\orho^2L^2(N+1)\big)\| B(z^*)\|^2
		+
8\Big\|\sumin w_i^*\Big\|^2
			\nonumber\\
		&\qquad
		+24\orho^2L^2 (N+1)N
+ 4N.
		\label{finalGradz}
\end{align}

\subsubsection{Dual Gradient Norm}
Considering that $\nabla\varphi_k$ is taken with respect to the subspace
$\cP$, the gradients with respect to the dual variables are (see for example
\cite{eckstein2009general}), for each $i=1,\ldots,n+1$,

\begin{align*}
	\|\nabla_{w_i}\varphi_k\|^2
	=
	\Big\|x_i^k - \frac{1}{n+1}\sum_{j=1}^{n+1} x_j^k\Big\|^2
	&=
	\Big\|\frac{1}{n+1}\sumjnp(x_i^k - x_j^k)\Big\|^2
	\\
	&\leq
	\sumjnp\|x_i^k - x_j^k\|^2
	\\
	&\leq
	2\sumjnp\left(\|x_i^k - z^k\|^2+\|z^k - x_j^k\|^2\right)
\end{align*}
Summing this inequality for $i=1,\ldots,n+1$ and collecting terms yields
\begin{align*}
\sum_{i=1}^{n+1}\|\nabla_{w_i}\varphi_k\|^2
&\leq
4(n+1)\sum_{i=1}^{n+1} \|x_i^k - z^k\|^2,
\end{align*}
so taking expectations conditioned on $\cF_k$ produces
\begin{align}
	\sum_{i=1}^{n+1}	\E[\|\nabla_{w_i}\varphi_k\|^2\,|\,\mbF_k]
	&\leq
	4(n+1)\suminp\E[\|x_i^k - z^k\|^2\,|\,\mbF_k]
	\nonumber\\
	&\leq
	4(n+1) \E[\|x_{n+1}^k - z^k\|^2\,|\,\mbF_k] +4(n+1)\sumin\E[\|x_i^k - z^k\|^2\,|\,\mbF_k]
	\nonumber\\
	&\leq
4(n+1) \E[\|x_{n+1}^k - z^k\|^2\,|\,\mbF_k]
\nonumber\\
&\qquad
+ 8(n+1)\sumin\E[\|x_i^k - z^*\|^2\,|\,\mbF_k] + 8(n+1)^2 \|z^k-z^*\|^2
	\nonumber\\
		&\leq
4(n+1) \E[\|x_{n+1}^k - z^k\|^2|\mbF_k]
\nonumber\\
&\qquad +8(n+1)\sumin\E[\|x_i^k - z^*\|^2|\mbF_k]
+8(n+1)^2 \|p^k-p^*\|^2
\nonumber	\\
	&\leq
	8(n+1)\big[3n+2\tau^2+1 + 9\orho^2(L^2+1)(N+1)\big]\|p^k - p^*\|^2
		\nonumber\\\label{dualGrad}
	&\quad\quad
	+72\orho^2(n+1)(N+1)\| B(z^*)\|^2
	+ 12\orho^2 (n+1)N,
\end{align}
where the final inequality employs \eqref{eqLast} and \eqref{eqinBone}.

All told, using \eqref{finalGradz} and \eqref{dualGrad} and simplifying the constants,
one obtains
\begin{align}
	\E[\|\nabla\varphi_k\|^2\,|\,\mbF_k]
	&= \E[\|\nabla_z\varphi_k\|^2\,|\,\mbF_k]
	+
	\suminp\E[\|\nabla_{w_i}\varphi_k\|^2 |\mbF_k]
	\nonumber\\
	&\leq C_1\|p^k - p^*\|^2 + C_2,
	\label{eqGradUpper}
\end{align}
where
\begin{align}
C_1
&=
24
		(1+10\orho^2)(n+1)(L^2+1)^2(N+1)^2
	\nonumber\\
	&\qquad
+
8(n+1)
\left(
2\tau^2
+
6(n+1) + 1
+
3(n+1)^2\tau^{-2}
\right)\label{defC1}
\end{align}
and
\begin{align}
C_2
&=
16(N+1)
\left[
1 +  4\orho^2(n+1) + 9\orho^2 L^2(N+1)
\right]
\|B(z^*)\|^2
+
8\|\sumin w_i^*\|^2
\nonumber\\
&
\qquad
+
12\orho^2 N
(
2 L^2(N+1) + n+1
)
+ 4N.
\label{defC2}
\end{align}

\subsection{Lower Bound for $\varphi_k$-gap}
Recalling \eqref{eqStart}, that is,
\begin{align*}
	\|p^{k+1} - p^*\|^2
	&=
	\|p^k - p^*\|^2 - 2\alpha_k(\varphi_k(p^k)-\varphi_k(p^*))+ \alpha_k^2\|\nabla\varphi_k\|^2.
\end{align*}
We may use the gradient bound from \eqref{eqGradUpper} to obtain
\begin{align}
\label{eqQuasi}
	\E[\|p^{k+1} - p^*\|^2 \,|\, \mbF_k]
	&\leq
	(1+C_1\alpha_k^2 )\|p^k - p^*\|^2 - 2\alpha_k\E[\varphi_k(p^k)-\varphi_k(p^*)\,|\,\mbF_k]
	+ C_2\alpha_k^2.
\end{align}
We now focus on finding a lower bound for the term
$\E[\varphi_k(p^k)-\varphi_k(p^*)\,|\,\mbF_k]$, which we call the
``$\varphi_k$-gap''.
Recall that for $p=(z,w_1,\ldots,w_{n+1})$,
\begin{align*}
	\varphi_k(p)
	&=
	\suminp\langle z - x_i^k,y_i^k - w_i\rangle .
\end{align*}
For each $i=1,\ldots,n+1$, define
$
\varphi_{i,k}(p)
\triangleq
\langle z - x_i^k,y_i^k - w_i\rangle.
$
We will call $\E[\varphi_{i,k}(p^k)-\varphi_{i,k}(p^*)\,|\,\mbF_k]$
the ``$\varphi_{i,k}$-gap''.  Note that
$\varphi_k(p) = \suminp\varphi_{i,k}(p)$.

\subsection{Lower Bound for $\varphi_{i,k}$-gap over $i=1,\ldots,n$}\label{secPhiiGap}
For $i=1,\ldots,n$, we have from line~\ref{yupdate} of the algorithm that
\begin{align*}
z^k - x_i^k = \tau(y_i^k - w_i^k).
\end{align*}
Since
$
\varphi_{i,k}(p^k)=\langle z^k - x_i^k,y_i^k - w_i^k\rangle,
$
one may conclude that for $i=1,\ldots,n$,
\begin{align*}
\varphi_{i,k}(p^k)=
\frac{\tau}{2}\|y_i^k - w_i^k\|^2 + \frac{1}{2\tau}\|z^k - x_i^k\|^2.
\end{align*}

On the other hand, for $p^*\in\cS$ and $i=1,\ldots,n$, one also has
\begin{align}
-\varphi_{i,k}(p^*)
=
\langle z^* - x_i^k,w_i^* - y_i^k\rangle
\geq 0\label{useMono1}
\end{align}
by the monotonicity of $A_i$.  Therefore, for $i=1,\ldots,n$, it holds that
\begin{align*}
\varphi_{i,k}(p^k) - \varphi_{i,k}(p^*)
\geq
\frac{\tau}{2}\|y_i^k - w_i^k\|^2 + \frac{1}{2\tau}\|z^k - x_i^k\|^2,
\end{align*}
and taking expectations conditioned on $\mbF_k$ leads to
\begin{align}
\E[\varphi_{i,k}(p^k) - \varphi_{i,k}(p^*)\,|\,\mbF_k]
\geq
\frac{\tau}{2}\|y_i^k - w_i^k\|^2+ \frac{1}{2\tau}\|z^k - x_i^k\|^2
\label{varphigapM}
\end{align}
where we have used that $x_i^k$ and $y_i^k$ are both $\mbF_k$-measurable for $i=1,\ldots,n$.

\subsection{Lower Bound for $\varphi_{n+1,k}$-gap}\label{secPhin_p_1_gap}
From lines~\ref{lineNoise1}-\ref{xupdateLip} of the algorithm, we have
\begin{align*}
z^k - x_{n+1}^k = \rho_k( B (z^k) - w_{n+1}^k + \epsilon^k).
\end{align*}
Therefore,
\begin{align}
\varphi_{n+1,k}(p^k)
&=
\langle z^k - x_{n+1}^k,y_{n+1}^k - w_{n+1}^k\rangle
\\
&=
\langle z^k - x_{n+1}^k, B(z^k) - w_{n+1}^k\rangle
+
\langle z^k - x_{n+1}^k,y_{n+1}^k -  B(z^k)\rangle
\nonumber\\
&=
\rho_k\langle  B(z^k) - w_{n+1}^k+\epsilon^k, B(z^k) - w_{n+1}^k\rangle
+
\langle z^k - x_{n+1}^k,y_{n+1}^k -  B(z^k)\rangle
\nonumber\\
&=
\rho_k\| B(z^k) - w_{n+1}^k\|^2
+
\langle z^k - x_{n+1}^k,y_{n+1}^k -  B(z^k)\rangle
+
\rho_k\langle\epsilon^k, B(z^k) - w_{n+1}^k\rangle
\nonumber\\
&\overset{\text{(a)}}{=}
\rho_k\| B (z^k) - w_{n+1}^k\|^2
+
\langle z^k - x_{n+1}^k , B(x_{n+1}^k)-  B(z^k)\rangle
+
\langle z^k - x_{n+1}^k ,e^k\rangle
\nonumber\\
&\qquad\qquad
+
\rho_k\langle \epsilon^k, B(z^k)-w_{n+1}^k\rangle
\nonumber\\
&\geq
\rho_k\| B (z^k) - w_{n+1}^k\|^2
-L\|z^k  - x_{n+1}^k\|^2
+
\langle z^k - x_{n+1}^k ,e^k\rangle
\nonumber\\
&\qquad\qquad
+
\rho_k\langle \epsilon^k, B(z^k)-w_{n+1}^k\rangle
\nonumber\\
&=
\rho_k\| B (z^k) - w_{n+1}^k\|^2
-L\|\rho_k( B(z^k) - w_{n+1}^k + \epsilon^k)\|^2
+
\langle z^k - x_{n+1}^k ,e^k\rangle
\nonumber\\
&\qquad\qquad
+
\rho_k\langle \epsilon^k, B(z^k)-w_{n+1}^k\rangle
\nonumber\\
&=
\rho_k\| B (z^k) - w_{n+1}^k\|^2
-\rho_k^2L\| B(z^k) - w_{n+1}^k + \epsilon^k\|^2
+
\langle z^k - x_{n+1}^k ,e^k\rangle
\nonumber\\
&\qquad\qquad
+
\rho_k\langle \epsilon^k, B(z^k)-w_{n+1}^k\rangle
\nonumber\\
&=
\rho_k(1-\rho_k L)\| B (z^k) - w_{n+1}^k\|^2
-\rho_k^2L\| \epsilon^k\|^2
+
\langle z^k - x_{n+1}^k ,e^k\rangle
\nonumber\\
&\qquad\qquad
+
\rho_k(1-2\rho_k L)\langle \epsilon^k, B(z^k)-w_{n+1}^k\rangle,
\label{lowerPhi1}
\end{align}
where equality (a) uses line~\ref{lineXYend} of the algorithm and the
inequality employs the Cauchy-Schwartz inequality followed by Lipschitz
continuity of $B$.

On the other hand,
\begin{align}
	-\varphi_{n+1,k}(p^*)
	&=
	\langle z^* - x_{n+1}^k, w_{n+1}^*-y_{n+1}^k\rangle
	\nonumber\\
	&=
	\langle z^* - x_{n+1}^k,  B(z^*)- B (x_i^k)\rangle
	+ \langle x_{n+1}^k-z^* ,e^k\rangle
	\nonumber\\\label{lowerPhi2}
	&\geq
	 \langle x_{n+1}^k-z^* ,e^k\rangle,
\end{align}
where the second equality uses line~\ref{lineXYend} of the algorithm and the
inequality follows from the monotonicity of $B$.

Combining  \eqref{lowerPhi1} and \eqref{lowerPhi2} yields
\begin{align}
\varphi_{n+1,k}(p^k) - \varphi_{n+1,k}(p^*)
&\geq
\rho_k(1-\rho_k L)\| B(z^k) - w_{n+1}^k\|^2
+
\rho_k(1-2\rho_k L)\langle \epsilon^k,B(z^k)-w_{n+1}^k\rangle
\nonumber\\
&\quad\quad
+
\langle z^k - x_{n+1}^k ,e^k\rangle
+
 \langle x_{n+1}^k-z^* ,e^k\rangle
 -\rho_k^2L\|\epsilon^k\|^2
 \nonumber\\
 &=
\rho_k(1-\rho_k L)\| B (z^k) - w_{n+1}^k\|^2
 -\rho_k^2L\|\epsilon^k\|^2
\nonumber\\
&\quad\quad
+
\rho_k(1-2\rho_k L)\langle \epsilon^k, B(z^k)-w_{n+1}^k\rangle
+
\langle z^k - z^* ,e^k\rangle.
\label{eqProg}
\end{align}

Now, if we take expectations conditioned on $\mbF_k$ and use \eqref{unbiasedAss}, we
obtain
\begin{align}\label{expect1}
\E\!\left[\left.\langle z^k - z^* ,e^k\rangle\,\right|\,\mbF_k\right]
&=
\langle z^k - z^* ,\E[e^k\,|\,\mbF_k]\rangle = 0.
\end{align}
Similarly, \eqref{unbiasedAss} also yields
\begin{align}\label{expect2}
\E\!\left[\left.\langle \epsilon^k, B(z^k)-w_{n+1}^k\rangle\,\right|\,\mbF_k\right]
&=
\left\langle \E[\epsilon^k|\mbF_k], B(z^k)-w_{n+1}^k\right\rangle
= 0.
\end{align}
Thus, using \eqref{expect1} and \eqref{expect2} and taking expectations of
\eqref{eqProg} yields
\begin{align}
\E[\varphi_{n+1,k}(p^k) -\varphi_{n+1,k}(p^*)\,|\,\mbF_k]
&\geq
\rho_k(1-\rho_k L)\| B (z^k) - w_{n+1}^k\|^2
 -\rho_k^2L\E[\|\epsilon^k\|^2|\mbF_k]
 \nonumber\\
&\geq
\rho_k(1-\orho L) \| B(z^k) - w_{n+1}^k\|^2
 -\rho_k^2N L(1+\| B(z^k)\|^2),
 \label{eqNew}
\end{align}
where in the second inequality we used \eqref{stepRuleSumInf} and the noise variance bound \eqref{noiseBound1}. Recall from \eqref{stepRuleSumInf} that $1-\orho L > 0$.

Next, we remark that
\begin{align*}
\| B(z^k)\|^2
& =
\| B(z^k)- B(z^*)+ B(z^*)\|^2
\\
&\leq
2L^2\|z^k-z^*\|^2+2\| B(z^*)\|^2
\leq
2 L^2\|p^k-p^*\|^2+2\| B(z^*)\|^2.
\end{align*}
Substituting this inequality into \eqref{eqNew} yields
\begin{align}
\E[\varphi_{n+1,k}(p^k) -\varphi_{n+1,k}(p^*)|\mbF_k]
&\geq
\rho_k (1-\orho L)\| B (z^k) - w_{n+1}^k\|^2
\nonumber\\
&\quad\quad
 -2\rho_k^2NL^3\|p^k-  p^*\|^2 - \rho^2_k N L (1+2\| B(z^*)\|^2).
 \label{varphigapL}
\end{align}

\paragraph{Finalizing the lower bound on the $\varphi_k$-gap}
Summing~\eqref{varphigapM} over $i=1,\ldots,n$ and using~\eqref{varphigapL} yields
\begin{align}
\E[\varphi_k(p^k) - \varphi_k(p^*)|\mbF_k]
&=
\suminp\E[\varphi_{i,k}(p^k) - \varphi_{i,k}(p^*)|\mbF_k]
\nonumber\\
&\geq
\frac{\tau}{2}\sumin\|y_i^k - w_i^k\|^2 + \frac{1}{2\tau}\sumin\|z^k - x_i^k\|^2
\nonumber\\
&\qquad\quad
+
\rho_k (1-\orho L)\| B (z^k) - w_{n+1}^k\|^2
 -2\rho_k^2 N L^3\|p^k-  p^*\|^2
\nonumber\\
&\qquad\quad
 - \rho^2_k N L(1+2\| B(z^*)\|^2).
\label{expectedPhi}
\end{align}

\subsection{Establishing Stochastic Quasi-Fejer Monotonicity}
Returning to \eqref{eqQuasi},
\begin{align*}
	\E[\|p^{k+1} - p^*\|^2 \,|\,\mbF_k]
	&\leq
	(1+C_1\alpha_k^2 )\|p^k - p^*\|^2 - 2\alpha_k\E[\varphi_k(p^k)-\varphi_k(p^*)\,|\,\mbF_k]
	+ C_2\alpha_k^2,
\end{align*}
we may now substitute \eqref{expectedPhi} for the expectation on the right-hand side. First, define
\begin{align*}
T_k\triangleq
\frac{\tau}{\orho}\sumin\|y_i^k - w_i^k\|^2 + \frac{1}{\orho\tau}\sumin\|z^k - x_i^k\|^2
+
2(1-\orho L)\| B (z^k) - w_{n+1}^k\|^2,
\end{align*}
after which we may use \eqref{expectedPhi} in \eqref{eqQuasi} to yield
\begin{align}
	\E[\|p^{k+1} - p^*\|^2 \,|\,\mbF_k]
	&\leq
	(1+ C_1\alpha_k^2+ C_3\alpha_k\rho_k^2)\|p^k - p^*\|^2 - \alpha_k \rho_k T_k
	+ C_2 \alpha_k^2 + C_4\alpha_k\rho_k^2\label{finalStochquas}
\end{align}
where $C_1$ and $C_2$ are defined as before in \eqref{defC1} and \eqref{defC2} and
\begin{align}\label{defC3}
C_3 &= 4NL^3
\\\label{defC4}
C_4 &=  2NL(1+2\| B(z^*)\|^2).
\end{align}
This completes the proof of Lemma \ref{lemKeyRecurse}.

\subsection{A Convergence Lemma}
Before establishing almost-sure convergence, we need the following lemma to
derive convergence of the iterates from convergence of $T_k$ defined above.
Note that a more elaborate result would be needed in an infinite-dimensional
setting.
\begin{lemma}\label{lemSeqConv}
For deterministic sequences $z^k\in\rR^{(n+1)d},\{(w_i^k)_{i=1}^{n+1}\}\in\cP$, and
$\{(x_i^k,y_i^k)_{i=1}^{n+1}\}\in\rR^{2(n+1)d}$, suppose that $y_i^k\in A_i(x_i^k)$ for $i=1,\ldots,n$,
$\suminp w_i^k = 0$,
\begin{align}\label{eqLemStart}
\xi_1\sumin\|y_i^k - w_i^k\|^2 + \xi_2\sumin\|z^k - x_i^k\|^2
+
\xi_3\| B (z^k) - w_{n+1}^k\|^2
\to 0
\end{align}
for scalars $\xi_1,\xi_2,\xi_3 > 0$, and $p^k\triangleq
(z^k,w_1^k,\ldots,w_{n+1}^k)\to\hat{p}
\triangleq(\hat{z},\hat{w}_1,\ldots,\hat{w}_{n+1})$.
Then $\hat{p}\in \cS$.
\end{lemma}
\begin{proof}
Fix any $i\in\{1,\ldots,n\}$.  Since $\|y_i^k - w_i^k\|\to 0$ by
\eqref{eqLemStart} and $w_i^k\to \hat{w}_i$, we also have $y_i^k\to\hat{w}_i$.
Similarly, \eqref{eqLemStart} also implies that $\|z^k -x_i^k\|\to 0$,
so from $z^k\to \hat{z}$ we also have $x_i^k\to
\hat{z}$. Since $y_i^k\in A_i(x_i^k)$ and $(x_i^k,y_i^k)\to
(\hat{z},\hat{w}_i)$, \cite[Prop.~20.37]{bauschke2011convex} implies
$\hat{w}_i\in A_i(\hat{z})$.  Since $i$ was arbitrary, the preceding
conclusions hold for $i=1,\ldots,n$.

Now, \eqref{eqLemStart} also implies that $\|B(z^k) - w_{n+1}^k\|\to 0$.
Therefore, since $w_{n+1}^k\to\hat{w}_{n+1}$, we also have $B(z^k)\to
\hat{w}_{n+1}$. Much as before, since
$(z^k,B(z^k))\to(\hat{z},\hat{w}_{n+1})$, we may apply
\cite[Prop.~20.37]{bauschke2011convex} to conclude that that $\hat{w}_{n+1} =
B(\hat{z})$.


Since the linear subspace $\cP$ defined in~\eqref{subspaceP} must be closed,
the limit $(\hat{z},\hat{w}_1,\ldots,\hat{w}_{n+1})$ of
$\{(z^k,w_1^k,\ldots,w_{n+1}^k)\} \subset \cP$ must be in $\cP$, hence
$\suminp
\hat{w}_i = 0$.

Thus, the point $\hat{p}=(\hat{z},\hat{w}_1,\ldots,\hat{w}_{n+1})$ satisfies
$\hat{w}_i\in A_i(\hat{z})$ for $i=1,\ldots,n$, $\hat{w}_{n+1}=B(\hat{z})$,
and $\suminp \hat{w}_i = 0$. These are the three conditions defining
membership in $\cS$ from \eqref{Sdef}, so $\hat{p}\in\cS$.
\end{proof}

\subsection{Finishing the Proof of Theorem \ref{thmMain} }
Given  $\sum_k\alpha_k^2<\infty$, and $\sum\alpha_k\rho_k^2<\infty$,
\eqref{finalStochquas} satisfies the conditions of \textit{Stochastic Quasi-Fejer Monotonicity} as given in Lemma \ref{lemQuasi}.  By applying Lemma \ref{lemQuasi}, we conclude
that there exist $\Omega_1,\Omega_2,\Omega_3$ such that $P[\Omega_i] = 1$ for $i=1,2,3$ and
\begin{enumerate}
\item for all $v\in\Omega_1$
\begin{align}\label{eqNotSums}
	\sum_{k=1}^\infty \alpha_k\rho_k T_k(v) <\infty,
\end{align}
\item
for all $v\in\Omega_2$, and $p^*\in\cS$, $\|p^k(v) - p^*\|$ converges to a finite nonnegative random-variable,
\item for all $v\in\Omega_3$, $p^k(v)$ remains bounded.
\end{enumerate}

Since $\sumk\alpha_k\rho_k = \infty$, \eqref{eqNotSums} implies that
for all $v\in\Omega_1$ there exists a subsequence $q_k(v)$ such that
\begin{align}\label{eqSubseqConv}
T_{q_k(v)} \to 0.
\end{align}


Let $\Omega'= \Omega_1\cap\Omega_2\cap \Omega_3$ and note that $P[\Omega']=1$.  Choose $v\in\Omega'$.
Since $p^k(v)$ remains bounded, so does $p^{q_k(v)}(v)$
for $q_k(v)$ defined above in \eqref{eqSubseqConv}. Thus
there exists a subsequence $r_k(v)\subseteq q_k(v)$ and $\hat{p}(v)\in \rR^{(n+2)d}$ such that $p^{r_k(v)}(v)\to \hat{p}(v)$.  But since $T_{q_k(v)}\to 0$, it also follows that $T_{r_k(v)}\to 0$, that is,
\begin{multline*}
\frac{\tau}{\orho}\sumin\|y_i^{r_k(v)}(v) - w_i^{r_k(v)}(v)\|^2 + \frac{1}{\orho\tau}\sumin\|z^{r_k(v)}(v) - x_i^{r_k(v)}(v)\|^2
\\
+
2(1-\orho L)\| B (z^{r_k(v)}(v)) - w_{n+1}^{r_k(v)}(v)\|^2
\to 0.
\end{multline*}
We then have from Lemma \ref{lemSeqConv} that $\hat{p}(v)\in\cS$.

Since $p^{r_k(v)}(v) \to  \hat{p}(v)$, it follows that $\|p^{r_k(v)}(v) -
\hat{p}(v)\|\to 0$. But since $\hat{p}(v)\in\cS$, $\|p^{k}(v) - \hat{p}(v)\|$
converges by point 2 above. Thus
\begin{align*}
\lim_{k\to\infty}\|p^k (v)- \hat{p}(v)\| = \lim_{k\to\infty}\|p^{r_k(v)}(v) - \hat{p}(v)\| = 0.
\end{align*}
Therefore $p^k(v)\to \hat{p}(v)\in \cS$. Thus
there exists $\hat{p}\in \cS$ such that $p^k\to\hat{p}$ a.s., which completes the proof of Theorem \ref{thmMain}.

\section{Proof of Lemma \ref{lemOk}}

If $O_k = 0$, then
\begin{align}
\label{O2}
\forall i =1,\ldots,n: \quad y_i^k = w_i^k
\text{ and } z^k = x_i^k.
\end{align}
Since $y_i^k\in A_i(x_i^k)$ for $i = 1,\ldots,n$,
\eqref{O2} implies that
that
\begin{align}\label{O4}
\forall i=1,\ldots,n: \quad w_i^k\in A_i(z^k).
\end{align}
Furthermore $O_k=0$ also implies that $w_{n+1}^k=B(z^k)$.
Finally, since $\suminp w_i^k=0$, we have that
$$
(z^k,w_1^k,\ldots,w_{n+1}^k)\in\cS.
$$

Conversely, suppose $(z^k,w_1^k,\ldots,w_{n+1}^k)\in \cS$.  The definition of
$\cS$ implies that $B(z^k) = w_{n+1}^k$ and furthermore that $w_i^k\in
A_i(z^k)$ for $i=1,\ldots,n$. For any $i=1,\ldots,n$, considering line
\ref{lineXYone} of Algorithm \ref{algSPS}, we may write $t_i^k = z^k+\tau
w_i^k\in (I+\tau A_i)(z^k)$, implying $z^k\in (I+\tau A_i)^{-1}(t_i^k)$.  But
since the resolvent $J_{\tau A_i}=(I+\tau A_i)^{-1}$ is single-valued
\cite[Prop.~23.8]{bauschke2011convex}, we must have $z^k = (I+\tau
A_i)^{-1}(t_i^k)$. Thus, by line \ref{xupdate}, we have $x_i^k = z^k$. We may
also derive from line \ref{yupdate} that
\begin{align*}
y_i^k = \tau^{-1}(t_i^k - x_i^k)
=
\tau^{-1}(z^k+\tau w_i^k - z^k)=w_i^k.
\end{align*}
Thus, since $x_i^k = z^k$ and $y_i^k = w_i^k$ for $i=1,\ldots, n$ and $w_{n+1}^k = B(z^k)$, we have that $O_k=0$.

\section{Proof of Theorem \ref{thmConvR}}

In addition to the proof, we provide a more detailed statement of the theorem:

\begin{theorem}
Fix the total iterations $K\geq 1$ of Algorithm \ref{algSPS} and
set
\begin{align}\label{step1s}
&\forall k=1,\dots, K:& \rho_k&=\rho\triangleq
\min
\left\{
K^{-1/4},\frac{1}{2L}
\right\}
\\
&\forall k=1,\dots, K:& \alpha_k &=
\alpha\triangleq
C_f \rho^2
\label{step2}
\end{align}
for some $C_f>0$.
Suppose
\eqref{unbiasedAss}-\eqref{noiseBound2} hold.
Then for any $p^*\in\cS$,
\begin{align}
\frac{1}{K}\sum_{j=1}^K
\E[O_j]
&\leq
\frac{8L^3\exp\left(C_f(C_1+C_3)\right)}{C_f\min\{\tau,\tau^{-1}\}K}
\left(
\|p^1 - p^*\|^2 + \frac{C_fC_2+C_4}{C_fC_1+C_3}\right)
& \text{for }
K &< (2L)^{4}
\label{eqRate1}
\\
\label{eqRate2}
\frac{1}{K}\sum_{j=1}^K
\E[O_j]
&\leq
\frac{\exp\left(C_f(C_1+C_3)\right)}{C_f\min\{\tau,\tau^{-1}\}K^{1/4}}
\left(
\|p^1 - p^*\|^2 + \frac{C_fC_2+C_4}{C_fC_1+C_3}\right)
& \text{for }
K &\geq (2L)^{4}.
\end{align}
where 
$O_k$ is the approximation residual defined in \eqref{Okdef}, and
$C_1,C_2,C_3,C_4$ are the nonegative constants defined in \eqref{defC1},
\eqref{defC2}, \eqref{defC3},  and \eqref{defC4}, respectively. Therefore,
\begin{align*}
\frac{1}{K}\sum_{j=1}^K
\E[O_j]
&=
\bigO(K^{-1/4}).
\end{align*}
\end{theorem}

\begin{proof}
Fix $\alpha_k=\alpha$ and $\rho_k = \rho$, where $\alpha$ and $\rho$ are the
respective right-hand sides of \eqref{step1s}-\eqref{step2}. Lemma
\ref{lemKeyRecurse} implies that \eqref{finalStochquasLem} so long as
\eqref{unbiasedAss}-\eqref{noiseBound2} hold and the stepsize $\rho$ satisfies
$\rho < L^{-1}$.
Since
\begin{align*}
\rho &=\min
\left\{
K^{-1/4},\frac{1}{2L}
\right\}
\leq
\frac{1}{2L},
\end{align*}
we conclude that~\eqref{finalStochquasLem} applies.

Rewriting \eqref{finalStochquasLem} with $\alpha_k=\alpha$ and $\rho_k=\rho$, we have
\begin{align*}
	\E[\|p^{k+1} - p^*\|^2 \,|\,\mbF_k]
	&\leq
				(1+C_1\alpha^2 +C_3\alpha\rho^2)\|p^k - p^*\|^2 - \alpha \rho T_k
+C_2\alpha^2 + C_4\alpha\rho^2.
\end{align*}
Therefore, taking expectations over $\cF_k$, we have
\begin{align}
	\E\|p^{k+1} - p^*\|^2
	&\leq
				(1+C_1\alpha^2 +C_3\alpha\rho^2)\E\|p^k - p^*\|^2 - \alpha \rho \E T_k
+C_2\alpha^2 + C_4\alpha\rho^2.
\label{finiteProg}
\end{align}
Recall that
\begin{align*}
T_k\triangleq
\frac{\tau}{\rho}\sumin\|y_i^k - w_i^k\|^2 + \frac{1}{\rho\tau}\sumin\|z^k - x_i^k\|^2
+
2(1-\orho L)\| B (z^k) - w_{n+1}^k\|^2,
\end{align*}
where for the first two terms we have simply set $\rho=\orho$ because the stepsize is constant. 
However, for the final term, we will still use an upper bound, $\orho$, on $\rho$. 
In the current setting, we know that $\rho \leq (1/2) L^{-1}$ and therefore
we may set $\orho = (1/2)L^{-1}$.  Thus
$1-\orho
L = 1/2$, leading to
\begin{align*}
\rho\E T_k
=
\tau\sumin\E\|y_i^k - w_i^k\|^2 + \tau^{-1}\sumin\E\|z^k - x_i^k\|^2
+
\rho \E\| B (z^k) - w_{n+1}^k\|^2.
\end{align*}
Let
\begin{align*}
U_k &\triangleq  \E\| B(z^k) - w_{n+1}^k\|^2
&
W_k&\triangleq
\tau\sumin\E\|y_i^k - w_i^k\|^2 + \tau^{-1}\sumin\E\|z^k - x_i^k\|^2,
\end{align*}
so that
\begin{align*}
\rho\E T_k = \rho U_k + W_k,
\end{align*}
and also let
\begin{align*}
	V_k &\triangleq \E\|p^{k} - p^*\|^2.
\end{align*}

Using these definitions in \eqref{finiteProg} we write
\begin{align*}
	V_{k+1} &\leq (1+C_1\alpha^2 + C_3\alpha\rho^2)V_k - \alpha\rho U_k - \alpha W_k + C_2 \alpha^2  + C_4\alpha\rho^2.
\end{align*}
Therefore,
\begin{align*}
	V_{k+1} + \alpha\rho U_k + \alpha W_k &\leq (1+C_1\alpha^2 + C_3\alpha\rho^2)V_k  + C_2\alpha^2 + C_4\alpha\rho^2
	\\
\iff	V_{k+1} + \alpha \rho\sum_{j=1}^{k }U_j +\alpha\sumj W_j
&\leq
(1+C_1\alpha^2 + C_3\alpha\rho^2)V_k
+\alpha \rho\sum_{j=1}^{k-1}U_j + \alpha\sumjm W_j
\\
&\qquad\quad
 + C_2 \alpha^2 +C_4\alpha\rho^2
\\
		&\leq
(1+C_1\alpha^2 + C_3\alpha\rho^2)
\left[V_k
+\alpha \rho\sum_{j=1}^{k-1}U_j + \alpha\sumjm W_j
\right]
\\
&\qquad\quad
 + C_2 \alpha^2 +C_4\alpha\rho^2,
\end{align*}
where we have used that $U_k,W_k\geq 0$.
Letting
\begin{align*}
R_k &= V_k + \alpha\rho\sum_{j=1}^{k-1}U_j + \alpha\sumjm W_j,
\end{align*}
we then have
\begin{align*}
R_{k+1}\leq (1+C_1\alpha^2 + C_3\alpha\rho^2)R_k + C_2\alpha^2 +C_4\alpha\rho^2,
\end{align*}
which implies
\begin{align*}
R_{k+1}\leq (1+C_1\alpha^2 + C_3\alpha\rho^2)^k R_1 + (C_2\alpha^2 +C_4\alpha\rho^2)\sum_{j=1}^k (1+C_1\alpha^2 + C_3\alpha\rho^2)^{k-j}.
\end{align*}
Now,
\begin{align*}
\sum_{j=1}^k (1+C_1\alpha^2+C_3\alpha\rho^2)^{k-j}
&=
\sum_{j=0}^{k-1} (1+C_1\alpha^2+C_3\alpha\rho^2)^{j}
\\
&=
\frac{(1+C_1\alpha^2+C_3\alpha\rho^2)^k-1}{(1+C_1\alpha^2+C_3\alpha\rho^2) - 1}
\\
&=
\frac{(1+C_1\alpha^2+C_3\alpha\rho^2)^k-1}{C_1\alpha^2+C_3\alpha\rho^2}
\\
&\leq
\frac{(1+C_1\alpha^2+C_3\alpha\rho^2)^k}{C_1\alpha^2+C_3\alpha\rho^2}.
\end{align*}
Therefore,
\begin{align*}
R_{k+1}\leq
(1+C_1\alpha^2+C_3\alpha\rho^2)^k
\left(R_1
+
\frac{C_2\alpha^2+C_4\alpha\rho^2}{C_1\alpha^2+C_3\alpha\rho^2}
 \right).
\end{align*}

Fix the number of iterations $K\geq 1$.
Now
\begin{align*}
\rho &= \min
\left\{
K^{-1/4},\frac{1}{2L}
\right\}
\leq
\frac{1}{K^{1/4}}
\leq 1.
\end{align*}
Therefore,
\begin{align*}
\alpha\rho\sum_{j=1}^{K}(U_j + W_j)
&\leq
\alpha\rho\sum_{j=1}^{K}U_j + \alpha\sum_{j=1}^{K} W_j
\\
&\leq R_{K+1}
\\
&\leq
 (1+C_1\alpha^2+C_3\alpha\rho^2)^K
\left(R_1
+
\frac{C_2\alpha^2+C_4\alpha\rho^2}{C_1\alpha^2+C_3\alpha\rho^2}
 \right).
\end{align*}

Dividing through by $\alpha\rho K$, we obtain
\begin{align}
\frac{1}{K}\sum_{j=1}^{K}(U_j +  W_j)
&\leq
  \frac{(1+C_1\alpha^2+C_3\alpha\rho^2)^K}{\alpha\rho K}
\left(R_1
+
\frac{C_2\alpha^2+C_4\alpha\rho^2}{C_1\alpha^2+C_3\alpha\rho^2}
 \right),\label{eqNearly}
\end{align}

and since $\alpha = C_f\rho^2$, we also have
\begin{align*}
\frac{C_2\alpha^2+C_4\alpha\rho^2}{C_1\alpha^2+C_3\alpha\rho^2}
=
\frac{C_fC_2+C_4}{C_fC_1+C_3}.
\end{align*}
Furthermore,
\begin{align*}
\rho \leq K^{-\frac{1}{4}}\implies\alpha \leq C_f K^{-\frac{1}{2}}.
\end{align*}
Substituting these into \eqref{eqNearly} yields
\begin{align}
\frac{1}{K}\sum_{j=1}^{K}(U_j +  W_j)
&\leq
\frac{\left(1+\frac{C_f (C_f C_1+C_3)}{K}\right)^K}{\alpha\rho K}
\left(
R_1 + \frac{C_f C_2+C_4}{C_f C_1+C_3}\right)
\nonumber\\
&\leq
\frac{\exp(C_f (C_f C_1+C_3))}{\alpha\rho K}
\left(
R_1 + \frac{C_f C_2+C_4}{C_f C_1+C_3}\right),
\label{asyProg}
\end{align}
where we have used that for any $t\geq 0$, $1+t/K\leq e^{t/K}$, so therefore
$(1+t/K)^K\leq e^t$.

The worst-case rates in terms of $K$ occur when $\rho = K^{-1/4}$ and $\alpha = C_f K^{-1/2}$.
This is the case when $K\geq (2L)^4$.
Substituting these into the denominator  yields, for $K\geq (2L)^4$, that
\begin{align*}
\frac{1}{K}\sum_{j=1}^{K}(U_j +  W_j)
&\leq
\frac{\exp(C_f (C_1+C_3))}{C_f K^{1/4}}
\left(
R_1 + \frac{C_f C_2+C_4}{C_f C_1+C_3}\right).
\end{align*}
Thus,
since 
$
O_k\leq \max\{\tau,\tau^{-1}\}\left(U_k + W_k\right),
$
we obtain
\begin{align*}
\frac{1}{K}\sum_{j=1}^K
\E[O_j]
&\leq
\frac{\max\{\tau,\tau^{-1}\}\exp\left(C_f(C_1+C_3)\right)}{C_fK^{1/4}}
\left(
\|p^1 - p^*\|^2 + \frac{C_f C_2+C_4}{C_f C_1+C_3}\right),
\end{align*}
which is \eqref{eqRate2}.

When $K< (2L)^4$, \eqref{eqRate1} can similarly be obtained
by substituting  $\rho = (2L)^{-1}$ and $\alpha = C_f (2L)^{-2}$ into
\eqref{asyProg}.
\end{proof}

\section{Approximation Residuals}
In this section we derive the approximation residual used to assess the performance of the algorithms in the numerical experiments. This residual relies on the following product-space reformulation of~\eqref{mono1}.
\subsection{Product-Space Reformulation}
Recall \eqref{mono1}, the monotone inclusion we are solving:
\begin{align*}
\text{Find }z\in\rR^d:
0 \in \sumin A_i(z) + B(z).
\end{align*}
In this section we demonstrate a ``product-space" reformulation of
\eqref{mono1} which allows us to rewrite it in a standard form involving just
two operators, one maximal monotone and the other monotone and Lipschitz. This
approach was pioneered in \cite{briceno2011monotone+,combettes2012primal}.
Along with allowing for a simple definition of an approximation residual as a
measure of approximation error in solving \eqref{mono1}, it allows for one to
apply operator splitting methods originally formulated for two operators to
problems such as~\eqref{mono1} for any finite $n$.

Observe that solving \eqref{mono1} is equivalent to
\begin{align*}
\text{Find }
(w_1,\ldots,w_n,z)\in\rR^{(n+1)d}:
w_i&\in A_i (z),\quad i=1,\ldots,n\\
0 &\in \sumin w_i + B(z).
\end{align*}
This formulation resembles that of the extended solution set $\cS$ used in
projective spitting, as given in~\eqref{Sdef}, except that it combines the
final two conditions in the definition of $\cS$, and thus does not need the
final dual variable $w_{n+1}$. Owing to the definition of the inverse of an
operator, the above formulation is equivalent to
\begin{align*}
\text{Find }
(w_1,\ldots,w_n,z)\in\rR^{(n+1)d}:
0&\in A_i^{-1}(w_i) - z,\quad i=1,\ldots,n\\
0 &\in \sumin w_i + B(z).
\end{align*}
These conditions are in turn equivalent to finding
$(w_1,\ldots,w_n,z)\in\rR^{(n+1)d}$ such that
\begin{align}
0
\in
\underbrace{
\big(A_1^{-1}(w_1)\times
A_2^{-1}(w_2)\times
\ldots\times
A_n^{-1}(w_n)\times
\{B(z)\}\big)
+
\left[
\begin{array}{cccc}
0 & \cdots & 0 & -I\\
\vdots & \ddots & \vdots & \vdots \\
0 & \cdots & 0 & -I\\
I & \cdots & I & 0
\end{array}
\right]
\left[
\begin{array}{c}
w_1\\
\vdots \\
w_n\\
z
\end{array}
\right]
}_{\triangleq \mathscr{T}(q)},
\label{monoProdSpace}
\end{align}
where $q = (w_1,\ldots,w_n,z)\in\rR^{(n+1)d}$. It may be shown, using
\cite[Proposition 20.23]{bauschke2011convex} and the fact that skew-symmetric
linear operators are monotone, that $\fT:\rR^{(n+1)d}\to 2^{\rR^{(n+1)d}}$ is
maximal monotone. Thus we have reformulated \eqref{mono1} as the monotone
inclusion: $0\in \fT(q)$ in the extended space $\rR^{(n+1)d}$. A vector
$z\in\rR^d$ solves \eqref{mono1} if and only if there exists
$(w_1,\ldots,w_n)\in \rR^{nd}$ such that $0\in \fT(q)$ where
$q=(w_1,\ldots,w_n,z)$.

For any pair $(q,v)$ such that $v\in \fT(q)$, $\|v\|$ represents an
\textit{approximation residual} for $q$ in the sense that $v=0$ implies $q$ is
a solution to \eqref{monoProdSpace}.  The norm $\|v\|$ is a measure of the
approximation error of $q$ as an approximate solution of \eqref{monoProdSpace}
and is only equal to $0$ at a solution. Given two approximate solutions $q_1$
and $q_2$ with certificates $v_1\in T(q_1)$ and $v_2\in \fT(q_2)$, we will
assume that $q_1$ is a \textit{better} approximate solution if
$\|v_1\|<\|v_2\|$.  This is somewhat analogous to the practice in optimization
of using the gradient $\|\nabla f(x)\|$ as a measure of quality of an
approximate minimizer of $f$.  However, note that since $\fT(q_1)$ is a set,
there may be elements of $\fT(q_1)$ with smaller norm than $v_1$. Thus any
given certificate only corresponds to an upper bound on
$\text{dist}(0,\fT(q_1))$.

\subsection{Approximation Residual for Projective Splitting}

In SPS (Algorithm \ref{algSPS}),
for $i=1,\ldots,n$, the pairs $(x_i^k,y_i^k)$ are chosen so that $y_i^k\in A_i(x_i^k)$.  This can be seen from the definition of the resolvent.  Thus $x_i^k\in A_i^{-1}(y_i^k)$.  Observe that for $\fT$ defined in \eqref{monoProdSpace}
\begin{align*}
v^k\triangleq
\left[
\begin{array}{c}
x_1^k - z^k\\
\vdots \\
x_n^k - z^k\\
B(z^k) + \sumin y_i^k
\end{array}
\right]
\in
\fT(y_1^k,\ldots,y_n^k,z^k).
\end{align*}
Thus $R_k\triangleq\|v^k\|^2$, defined in \eqref{defRk}, represents a measure
of the approximation error for SPS, in the sense that $v^k=0$ implies  $z^k$
solves \eqref{mono1}. We may relate $R_k$ to the approximation residual $O_k$
for SPS from Section \ref{secConvRate} as follows:
\begin{align*}
R_k
&=
\sumin \|z^k - x_i^k\|^2
+\left\|B(z^k) +\sumin y_i^k\right\|^2
\\
&=
\sumin \|z^k - x_i^k\|^2
+\left\|B(z^k) +\sumin y_i^k - \suminp w_i^k\right\|^2
\\
&\leq
\sumin \|z^k - x_i^k\|^2
+
2\|B(z^k) - w_{n+1}^k\|^2
+
2\left\|\sumin (y_i^k - w_i^k)\right\|^2
\\
&\leq
\sumin \|z^k - x_i^k\|^2
+
2\|B(z^k) - w_{n+1}^k\|^2
+
2n\sumin\left\|y_i^k - w_i^k\right\|^2
\\
&\leq
2n O_k
\end{align*}
where in the second equality we have used the fact that $\suminp w_i^k = 0$.
Thus $R_k$ has the same convergence rate as $O_k$ given in Theorem \ref{thmConvR}.

\subsection{Approximation Residual for Tseng's method}
\newcommand{\fA}{\mathscr{A}}
\newcommand{\fB}{\mathscr{B}}

 Tseng's method \cite{tseng2000modified} can be applied to \eqref{monoProdSpace}  resulting in the following iteration, applied in the product space
with $q^k \in \rR^{(n+1)d}$,
\begin{align}
\bar{q}^k &=
J_{\alpha \fA}(q^k - \fB(q^k))
\label{Tseng1}
\\
q^{k+1}
&=
\bar{q}^k + \alpha(\fB(q^k) - \fB(\bar{q}^k))
\label{Tseng2}
\end{align}
where
\begin{align}
\fA(w_1,\ldots,w_n,z)\mapsto
(A_1^{-1}(w_1)\times
A_2^{-1}(w_2)\times
\ldots\times
A_n^{-1}(w_n)\times
\{0\})
\label{defFancyA}
\end{align}
and
\begin{align}
\fB(w_1,\ldots,w_n,z)\mapsto
\left[
\begin{array}{cccc}
0 & \cdots & 0 & -I\\
\vdots & \ddots & \vdots & \vdots \\
0 & \cdots & 0 & -I\\
I & \cdots & I & 0
\end{array}
\right]
\left[
\begin{array}{c}
w_1\\
\vdots \\
w_n\\
z
\end{array}
\right]
+
\big(\{0\} \times \cdots \times \{0\} \times B(z)\big).
\label{defFancyB}
\end{align}
Note that $\fT = \fA + \fB$. The operator $\fB$ may be shown to be monotone
and Lipschitz, while $\fA$ is maximal monotone. The resolvent of $\fA$ may be
readily computed from the resolvents of $A_i$ using Moreau's identity
\cite[Proposition 23.20]{bauschke2011convex}.

Analogous to SPS, Tseng's method has an approximation residual, which in this
case is an element of $\fT(\bar{q}^k)$. In particular, using the general
properties resolvent operators as applied to $J_{\alpha\fA}$, we have
\begin{align*}
\frac{1}{\alpha}(q^k - \bar{q}^k)
- \fB(q^k) \in \fA(\bar{q}^k).
\end{align*}
Also, rearranging \eqref{Tseng2} produces
\begin{align*}
\frac{1}{\alpha}(\bar{q}^k - q^{k+1}) + \fB (q^k)
=
\fB(\bar{q}^k).
\end{align*}
Adding these two relations produces
\begin{align*}
\fT(\bar{q}^k)
=
\fA(\bar{q}^k)
+
\fB(\bar{q}^k)
\ni
\frac{1}{\alpha}(q^k - q^{k+1}).
\end{align*}
Therefore,
\begin{align*}
R^{\text{Tseng}}_k \triangleq \frac{1}{\alpha^2}\|q^k - \bar{q}^{k+1}\|^2
\end{align*}
represents a measure of the approximation error for Tseng's method
equivalent to $R_k$ defined in \eqref{defRk} for SPS.

\subsection{Approximation Residual for FRB}
The forward-reflected-backward method (FRB) \cite{malitsky2020forward} is
another method that may be applied to the splitting $\fT = \fA + \fB$ for
$\fA$ and $\fB$ as defined in \eqref{defFancyA} and \eqref{defFancyB}. Doing so
yields the following method
\begin{align*}
q^{k+1}
=
J_{\alpha\fA}
[q^k
-
\alpha(2\fB(q^k) - \fB(q^{k-1}))
].
\end{align*}
Following similar arguments to those for Tseng's method, it can be shown that
\begin{align*}
\frac{1}{\alpha}
\left(
q^{k-1}
-q^k
\right)
+
\fB(q^k)
+
\fB(q^{k-2})
-
2\fB(q^{k-1})
\triangleq
v_{\text{FRB}}^k
\in
\fT(q^k).
\end{align*}
Thus, FRB admits the following approximation residual equivalent to $R_k$ for SPS:
\begin{align*}
R^{\text{FRB}}_k\triangleq \|v_{\text{FRB}}^k\|^2.
\end{align*}

To summarize, Figure \ref{fig} plots $R_k$ for SPS, $R^{\text{Tseng}}_k$ for
Tseng's method, and $R^{\text{FRB}}_k$ for FRB.

Finally, we point out that the stepsizes used in both Tseng and FRB can be chosen via a linesearch procedure which we do not detail here.

\section{Variational Inequalities}

For a mapping $B:\rR^d\to \rR^d$ and a closed and convex set $\cC$, the variational inequality problem \cite{harker1990finite} is to find $z^*\in\cC$ such that
\begin{align}\label{defVI}
B(z^*)^\top (z - z^*) \geq 0, \forall z\in \cC.
\end{align}
Consider the normal cone mapping discussed in Section \ref{secBackG} and defined as
\begin{align*}
N_{\cC}(x) \triangleq \{g: g^\top(y-x)\leq 0\,\, \forall y\in \cC\}
\end{align*}
It is easily seen that \eqref{defVI} is equivalent to finding $z^*$ such that
$-B(z^*)\in N_\cC(z^*)$. Hence, if $B$ is monotone, \eqref{defVI}
is equivalent to the monotone inclusion
\begin{align}\label{vi}
0\in B(z^*) + N_\cC(z^*).
\end{align}
Thus, monotone variational inequalities are a special case of monotone
inclusions with two operators, one of which is single-valued and the other is
the normal cone map of the constraint set $\cC$. As a consequence, methods for
monotone inclusions can be used to solve monotone variational inequality
problems. The reverse, however, may not be true. For example, the analysis of
the extragradient method~\cite{korpelevich1977extragradient} relies on the
second operator $N_{\cC}$ in~\eqref{vi} being a normal cone, as opposed to a
more general monotone operator.  We are not aware of any direct extension of
the extragradient method's analysis allowing a more general resolvent to be
used in place of the projection map corresponding to $N_{\cC}$.

\section{Memory-Saving Technique for SPS}
The variables $t_i^k$, $x_i^k$, and $y_i^k$ on lines
\ref{lineXYone}-\ref{yupdate} of SPS are stored in variables $t,x$ and $y$.
Another two variables $\bar{x}$ and $\bar{y}$ keep track of $\sumin x_i^k$ and
$\sumin y_i^k$. The dual variables are stored as $w_i$ for $i=1,\ldots,n$ and
the primal variable as $z$. Once $x=x_i^k$ is computed, the $i^\text{th}$ dual
variable $w_i$ can be partially updated as $w_i \leftarrow w_i-\alpha_k x$.
Once all the operators have been processed,  the update for each dual variable
may be completed via $w_i \leftarrow w_i+\alpha_k(n+1)^{-1}\bar{x}$. Also, the
primal update is computed as $z \leftarrow z-\alpha_k\bar{y}$. During the
calculation loop for the $x_i^k, y_i^k$, the terms in approximation residual
$R_k$ may also be accumulated one by one.  The total total number of vector
elements that must be stored is $(n
+ 7)d$.

\section{Additional Information About the Numerical Experiments}
Recall the problem~\eqref{drslr} considered in the numerical experiments:
\begin{align}
\begin{array}{rl}
\displaystyle{\min_{\substack{\beta\in\rR^d \\ \lambda\in\rR\,\,\,}}} \;\;
\displaystyle{\max_{\gamma\in\rR^m}}
&
\displaystyle{
\left\{
\lambda(\delta - \kappa) +
\frac{1}{m}\sum_{i=1}^m\Psi(\langle \hat{x}_i,\beta\rangle)
+
\frac{1}{m}
\sum_{i=1}^m
\gamma_i(
\hat{y}_i\langle\hat{x}_i,\beta\rangle - \lambda\kappa
)
+
c\|\beta\|_1
\right\}
}
\\
\,\text{s.t.} &
\|\beta\|_2\leq \lambda/(L_\Psi+1) \qquad \|\gamma\|_\infty\leq 1.
\end{array}
\label{drslr2}
\end{align}
We now show how we converted this problem to the form~\eqref{mono1} for our experiments.
Let $z$ be a shorthand for $(\lambda,\beta,\gamma)$
and define
\begin{align*}
\mathcal{L}(z)\triangleq
\lambda(\delta - \kappa) +
\frac{1}{m}\sum_{i=1}^m\Psi(\langle \hat{x}_i,\beta\rangle)
+
\frac{1}{m}
\sum_{i=1}^m
\gamma_i(
\hat{y}_i\langle\hat{x}_i,\beta\rangle - \lambda\kappa
).
\end{align*}
The first-order necessary and sufficient conditions for the convex-concave saddlepoint problem in \eqref{drslr2} are
\newcommand\sample[1]{\langle\hat{x}_{#1},\beta\rangle}
\begin{align} \label{monon2}
0
\in
B(z)
+
A_1(z)
+
A_2(z)
\end{align}
where the vector field $B(z)$ is defined as
\begin{align}\label{defB}
B(z)
\triangleq
\left[
\begin{array}{c}
\nabla_{\lambda,\beta} \mathcal{L}(z)\\
-\nabla_{\gamma} \mathcal{L}(z)
\end{array}
\right],
\end{align}
with
\begin{align*}
\nabla_{\lambda,\beta} \mathcal{L}(z)
=
\left[
\begin{array}{c}
\delta - \kappa(1+\frac{1}{m}\sum_{i=1}^m\gamma_i)\\
\frac{1}{m}\sum_{i=1}^m\Psi'(\sample{i})\hat{x}_i
+\frac{1}{m}\sum_{i=1}^m\gamma_i\hat{y}_i\hat{x}_i
\end{array}
\right]
\end{align*}
and
\begin{align*}
\nabla_\gamma \mathcal{L}(z)
=
\left[
\begin{array}{c}
\frac{1}{m}(\hat{y}_1\sample{1}-\lambda\kappa)
\\
\vdots
\\
\frac{1}{m}(\hat{y}_m\sample{m}-\lambda\kappa)
\end{array}
\right].
\end{align*}
It is readily confirmed that $B$ defined in this manner is Lipschitz.
\col{Monotonicity of $B$ follows from the fact that it is the generalized gradient of a convex-concave saddle function \cite{rockafellar1970monotone}}.
For the set-valued operators, $A_1(z)$ corresponds to the constraints and $A_2(z)$ to the nonsmooth $\ell_1$ regularizer, and are defined as
\begin{align*}
A_1(z)
\triangleq
N_{\cC_1}(\lambda,\beta)\times N_{\cC_2}(\gamma),
\end{align*}
where
\begin{align*}
\cC_1
\triangleq
\big\{
(\lambda,\beta):
\|\beta\|_2\leq \lambda/(L_\Psi+1)
\big\}
\quad
\text{ and }
\quad
\cC_2\triangleq
\{\gamma:
\|\gamma\|_\infty\leq 1
\},
\end{align*}
and
\begin{align*}
A_2(z)
\triangleq
\{\mathbf{0}_{1\times 1}\}\times
c\partial \|\beta\|_1
\times\{\mathbf{0}_{m\times 1}\}.
\end{align*}
Here, the notation $\mathbf{0}_{p\times 1}$ denotes the $p$-dimensional vector
of all zeros. $\cC_1$ is a scaled version of the second-order cone, well known
to be a closed convex set, while $\cC_2$ is the unit ball of the $\ell_\infty$ norm,
also closed and convex.  Since $A_1$ is a normal cone map of a closed convex
set and $A_2$ is the subgradient map of a closed proper convex function (the
scaled $1$-norm), both of these operators are maximal monotone and
problem~\eqref{monon2} is a special case of~\eqref{mono1} for $n=2$.

\paragraph{Stochastic oracle implementation}
The operator $B:\rR^{m+d+1}\mapsto\rR^{m+d+1}$, defined in \eqref{defB},
can be written as
\begin{align*}
B(z)
=
\frac{1}{m}\sum_{i=1}^m B_i(z)
\end{align*}
where
\begin{align*}
B_i(z)
\triangleq
\left[
\begin{array}{c}
\delta - \kappa(1+\gamma_i)\\
\Psi'(\sample{i})\hat{x}_i
+\gamma_i\hat{y}_i\hat{x}_i
\\
\mathbf{0}_{(i-1)\times 1}
\\
-(\hat{y}_i\sample{i}-\lambda\kappa)
\\
\mathbf{0}_{(m - i)\times 1}
\end{array}
\right].
\end{align*}
In our SPS experiments, the stochastic oracle for $B$ is simply $\tilde{B}(z)
= \frac{1}{|\mathbf{B}|}\sum_{i\in \mathbf{B}} B_i(z)$ for some minibatch
$\mathbf{B}\subseteq\{1,\ldots,m\}$. We used a batchsize of $100$.

\paragraph{Resolvent computations}
The resolvent of $A_1$ is readily constructed from the projection maps of
the simple sets $\cC_1$ and $\cC_2$, while the resolvent $A_2$ involves the
proximal operator of the $\ell_1$ norm.  Specifically,
\begin{align*}
J_{ \rho A_1}(z)
=
\left[
\begin{array}{c}
\proj_{\cC_1}\!(\lambda,\beta)\\
\proj_{\cC_2}\!(\gamma)
\end{array}
\right]
\quad
\text{ and}
\quad
J_{\rho A_2}(z)
=
\left[
\begin{array}{c}
\mathbf{0}_{1\times 1}\\
\prox_{\rho c\|\cdot\|_1}\!(\beta)\\
\mathbf{0}_{m\times 1}
\end{array}
\right].
\end{align*}
The constraint $\cC_1$ is a scaled second-order cone and $\cC_2$ is the
$\ell_\infty$ ball, both of which have closed-form projections. The proximal
operator of the $\ell_1$ norm is the well-known soft-thresholding operator
\cite[Section 6.5.2]{parikh2013proximal}.  Therefore all resolvents in the
formulation may be computed quickly and accurately.

\paragraph{SPS stepsize choices}
For stepsize in SPS, we ordinarily require $\rho_k \leq\orho< 1/L$ for the
global Lipschitz constant $L$ of $B$. However, since the global Lipschitz
constant may be pessimistic, better performance can often be achieved by
experimenting with larger stepsizes.  If divergence is observed, then the
stepsize can be decreased. This type of strategy is common for SGD and similar
stochastic methods.
Thus, for SPS-decay we set
$
\alpha_k = C_d k^{-0.51}
$
and
$
\rho_k = C_d k^{-0.25},
$
and experimented with different values for $C_d$. For SPS-fixed we used $\rho=
K^{-1/4}$ and $\alpha = C_f\rho^2$, and experimented with different values for
$C_f$. The total number of iterations for SPS-fixed was chosen as follows: For
the epsilon dataset, we used $K=5000$, for SUSY we used $K=200$,  and for
real-sim we used $K=1000$.

\paragraph{Parameter choices for the other algorithms}
For Tseng's method, we used the backtracking linesearch variant with an
initial stepsize of $1$, $\theta=0.8$, and a stepsize reduction factor of
$0.7$. For FRB, we used the backtracking linesearch variant with the same
settings as for Tseng's method. For deterministic PS, we used a fixed stepsize
of $0.9/L$.

\section{Local Convergence on Non-Monotone Problems}

The work \cite{NEURIPS2020_ba9a56ce} provides a local convergence analysis for
DSEG applied to locally monotone problems.  Recall that DSEG is equivalent
to the special case of SPS for which $n=0$. While extending this result to
the more general setting of SPS is beyond the scope of this manuscript, we next provide a
preliminary sketch of how the analysis of~\cite{NEURIPS2020_ba9a56ce} might be
generalized to our setting.  We leave a formal proof to future work.

\newcommand{\Brz}{\mathbb{B}_r(z^*)}
\paragraph{Sketch of assumptions and main result}
The first assumption needed is the existence of an isolated solution
$p^*=(z^*,w_1^*,\ldots,w_{n+1}^*)\in\cS$. We then require that there exists a
ball $\mathbb{B}_r(z^*)$, centered at $z^*$, throughout which the operator $B$
is ``well-behaved'', meaning that it satisfies monotonicity and Lipschitz
continuity.  In addition, we need each $A_i$, for $i=1,\ldots,n$, to be
maximal monotone within this ball.  Outside of the ball, the operators do not
need to be monotone or Lipschitz.

Following \cite[Assumption $2'$]{NEURIPS2020_ba9a56ce}, the noise variance 
assumptions are slightly stronger than in the monotone case.   In particular,
we require that
$
\E[\|\epsilon^k\|^q|\mbF_k]\leq N^q
$
and
$
\E[\|e^k\|^q|\mbF_k]\leq N^q
$
for some $q>2$.
As before,  the noise must be zero-mean.
Finally,  the stepsize requirements are also slightly stronger than 
\eqref{stepRuleSumInf},  having the added assumption that
$\sumk\rho_k^q<\infty$.

With these assumptions, the goal is to show that, so long as the initial point
$p^1$ is sufficiently close to $p^*$, then with high probability $p^k$
converges to $p^*$.

\paragraph{Proof strategy}
The initial strategy is to develop the following recursion, satisfied by SPS, that does not (yet) utilize local monotonicity or Lipschitz continuity:
\allowdisplaybreaks
\begin{align}
 \|p^{k+1} - p^*\|^2
&\leq
(1+c_1\alpha_k^2) \|p^k - p^*\|^2
-
c_2\alpha_k\rho_k (T_k'
+ l_k
+ r_k
)
-c_3\alpha_k (r_k' + q_k)
\nonumber\\\label{nc_recursion}
&\qquad
+
c_1\alpha_k^2\big(\|e^k\|^2
+
 \|\epsilon^k\|^2
 +
 c_4
 \big)+
 c_5\alpha_k q'_k
\end{align}
for appropriate constants $c_1\ldots c_5\geq 0$.  In this inequality, we use
\begin{align*}
T_k'
&\triangleq
\frac{\tau}{\orho}\sumin\|y_i^k - w_i^k\|^2 + \frac{1}{\orho\tau}\sumin\|z^k - x_i^k\|^2,
\\
l_k &\triangleq \sum_{i=1}^{n}\langle z^* - x_i^k, w_i^* - y_i^k\rangle
+
\big\langle z^* - x_{n+1}^k, w_i^* - B (x_{n+1}^k)\big\rangle,
\\
r_k &\triangleq
\langle \epsilon^k, B(\tilde{x}^k)- w_{n+1}^k\rangle
,
\\
r_k' &\triangleq
\langle z^k - z^*,e^k\rangle,\quad
\\
q_k &\triangleq
(\rho_k^{-1}-d/2)\|\tilde{x}^k - z^k\|^2
-
\|\tilde{x}^k - z^k\|\|B(\tilde{x}^k) - B(z^k)\|
\\
q'_k 
&\triangleq
\rho_k\|\epsilon^k\|\|B x_{n+1}^k-B\tilde{x}^k\|
+
\frac{1}{2d}\|B\tilde{x}_{n+1}^k - B x_{n+1}^k\|^2,
\end{align*}
where
\begin{align}\label{defXtilde}
\tilde{x}^k &\triangleq z^k - \rho_k\big(B(z^k) - w_{n+1}^k\big) &
d&\triangleq \frac{1-\orho L}{1+\orho/2},
\end{align}
with $L$ being the local Lipschitz constant of $B$ on $\Brz$.  The iterate
$\tilde{x}^k$ is the analog of the iterate $\tilde{X}_{t+1/2}$ used in
\cite{NEURIPS2020_ba9a56ce}. 

The recursion \eqref{nc_recursion} is derived by once again starting from
\eqref{eqStart} and following the arguments leading to \eqref{eqQuasi}, but
this time not taking conditional expectations. In particular, the upper bounds
on $\|\nabla_z\varphi_k\|^2$ and $\|\nabla_{w_i}\varphi_k\|^2$ contribute the
terms $c_1\alpha_k^2\big(\|e^k\|^2
+
 \|\epsilon^k\|^2
 +c_4)
$
and
$
c_1\alpha_k^2\|p^k - p^*\|^2
$.
For $i=1,\ldots,n$, the ``$\varphi_{i,k}$-gap" term, $\varphi_{i,k}(p^k) - \varphi_{i,k}(p^*)$, 
is dealt with in a similar manner to Section \ref{secPhiiGap}, but this time not using monotonicity as in \eqref{useMono1}. This
contributes $T'_k$ and the first term in $l_k$.  
Finally, as we sketch below, the ``$\varphi_{n+1,k}$-gap" term contributes
 $r_k$, $r'_k$, $q_k$,  $q'_k$, and the last term in $l_k$.
 
For the ``$\varphi_{n+1,k}$-gap'', that is,  $\varphi_{n+1,k}(p^k) -
\varphi_{n+1,k}(p^*)$,  we have to depart from the analysis in Section
\ref{secPhin_p_1_gap} and use an alternative argument involving $\tilde{x}^k$.
We now provide some details of this argument:  in the following,  we use $Bz$
as shorthand for $B(z)$ for any vector $z\in\rR^d$. We begin the analysis with
\begin{align}
\varphi_{n+1,k}(p^k)&=
\langle z^k - x_{n+1}^k,y_{n+1}^k - w_{n+1}^k\rangle 
\nonumber\\
&=
\langle z^k - x_{n+1}^k,B x_{n+1}^k - w_{n+1}^k\rangle 
+
\underbrace{\langle z^k - x_{n+1}^k,e^k\rangle}_{\text{part of }r'_k}. \label{chain1}
\end{align}
The final term will combine with the term $\langle x_{n+1}^k - z^*,e^k\rangle$ coming from  
\begin{align}
-\varphi_{n+1,k}(p^*)
&=
\langle
z^*-x_{n+1}^k
,
w_{n+1}^*-y_{n+1}^k
\rangle 
\nonumber\\\label{bottomPhi}
&=
\langle
 z^*-x_{n+1}^k
,
w_{n+1}^*-Bx_{n+1}^k
\rangle
+
\langle
x_{n+1}^k - z^*
,
e_{n+1}^k
\rangle
\end{align}
to yield $r_k'$ above.  Equation \eqref{bottomPhi} also yields the second term
in $l_k$. Using that $\tilde{x}^k - x_{n+1}^k = \rho_k\epsilon_k$, we
rewrite the first term in~\eqref{chain1} as
\begin{align}
\langle z^k - x_{n+1}^k,B x_{n+1}^k - w_{n+1}^k\rangle 
&=
\langle z^k - \tilde{x}^k,B x_{n+1}^k - w_{n+1}^k\rangle 
+
\langle \tilde{x}^k - x_{n+1}^k,B x_{n+1}^k - w_{n+1}^k\rangle 
\nonumber\\
&=
\langle z^k - \tilde{x}^k,B x_{n+1}^k - w_{n+1}^k\rangle 
+
\rho_k\langle \epsilon^k,B x_{n+1}^k - w_{n+1}^k\rangle 
\nonumber\\
&=
\langle z^k - \tilde{x}^k,B x_{n+1}^k - w_{n+1}^k\rangle 
+
\rho_k\langle \epsilon^k,B x_{n+1}^k - B \tilde{x}^k\rangle 
\label{chain2}\\
&\qquad+
\rho_k\underbrace{\langle \epsilon^k,B \tilde{x}^k - w_{n+1}^k\rangle}_{r_k}.
\nonumber
\end{align}
Next, the terms in~\eqref{chain2}
admit the lower bound
\begin{multline*}
\langle z^k - \tilde{x}^k,B x_{n+1}^k - w_{n+1}^k\rangle 
+
\rho_k\langle \epsilon^k,B x_{n+1}^k-B\tilde{x}^k\rangle 
\\
\geq 
\langle z^k - \tilde{x}^k,B x_{n+1}^k - w_{n+1}^k\rangle 
-
\underbrace{\rho_k\|\epsilon^k\|\|B x_{n+1}^k-B\tilde{x}^k\|}_{\text{first part of }q'_k}.
\end{multline*}
Considering the first term on right-hand side of this bound, we also have
\begin{align*}
\langle z^k - \tilde{x}^k,B x_{n+1}^k - w_{n+1}^k\rangle 
&=
\langle z^k - \tilde{x}^k,B \tilde{x}^k - w_{n+1}^k\rangle 
+
\langle z^k - \tilde{x}^k,Bx_{n+1}^k - B \tilde{x}^k\rangle 
\\
&\geq 
\langle z^k - \tilde{x}^k,B \tilde{x}^k - w_{n+1}^k\rangle 
- 
\frac{d}{2}\|z^k - \tilde{x}^k\|^2
-
\underbrace{\frac{1}{2d}\|B\tilde{x}^k - B x_{n+1}^k\|^2}_{\text{second part of }q'_k}
\end{align*}
for any $d>0$, using Young's inequality. 
Finally, for the first two terms of the right-hand side of the above relation,
we may write
\begin{multline*}
\langle z^k - \tilde{x}^k,B \tilde{x}^k - w_{n+1}^k\rangle 
- 
\frac{d}{2}\|z^k - \tilde{x}^k\|^2
\\
=
\langle z^k - \tilde{x}^k,B z^k - w_{n+1}^k\rangle 
+
\langle z^k - \tilde{x}^k, B\tilde{x}^k - B z^k\rangle 
- 
\frac{d}{2}\|z^k - \tilde{x}^k\|^2
\\
\geq 
\underbrace{(\rho^{-1}_k-d/2)\|z^k - \tilde{x}^k\|^2
-
\|z^k - \tilde{x}^k\|\|B\tilde{x}^k - B z^k\|}
_{q_k},
\end{multline*}
where in the final inequality we use the Cauchy-Schwartz inequality and
substitute $Bz^k - w_{n+1}^k = \rho_k^{-1}(z^k - \tilde{x}^k)$, from the
definition of $\tilde{x}^k$ in \eqref{defXtilde}.  We have now accounted for
all the terms appearing in~\eqref{nc_recursion}.

The recursion \eqref{nc_recursion}  is analogous to equation (F.7) on page 24
of \cite{NEURIPS2020_ba9a56ce} and provides the starting point for the local
convergence analysis. The next step would be to derive an analog of Theorem
F.1.~of \cite{NEURIPS2020_ba9a56ce} using~\eqref{nc_recursion}. 
The following translation to the notation of Theorem F.1.~could be used (note that \cite{NEURIPS2020_ba9a56ce} uses $t$ for iteration counter):
\begin{align*}
 D_k &= \|p^k - p^*\|^2,\\
\zeta_k &= c_2\alpha_k\rho_k (T'_k+l_k)+c_3\alpha_k q_k,\\
\xi_k &= -c_2\alpha_k\rho_k r_k-c_3\alpha_k r_k',\\
\chi_k &=
c_1\alpha_k^2\big(\|e^k\|^2
+
 \|\epsilon^k\|^2
+\|p^k-p^*\|^2 
+
c_4 
 \big)
 +c_5\alpha_k q'_k
 ,
\end{align*}
and the event $E_\infty^\rho$ is translated to
\begin{align*}
E_\infty^\rho
=
\left\{
x_{n+1}^k\in\Brz,
\tilde{x}^k\in\mathbb{B}_{\rho r}(z^*),
p^k\in\mathbb{B}_{\rho r}(p^*)
\text{ for all }k=1,2,\ldots
\right\}.
\end{align*} 
An analog of
Theorem 2 of \cite{NEURIPS2020_ba9a56ce} could then be developed based on this result.


\end{document}